\documentclass[reqno]{amsart}
\usepackage{amsmath, amsthm, amssymb, color}
\usepackage{graphicx}
\usepackage{mathrsfs}
\usepackage{latexsym}
\usepackage{hyperref}
\usepackage{xcolor}
\numberwithin{equation}{section}
\newtheorem{Theorem}{Theorem}[section]
\newtheorem{Lemma}[Theorem]{Lemma}

\theoremstyle{definition}

\newtheorem{remark}[Theorem]{Remark}
\newtheorem{Proposition}[Theorem]{Proposition}

\providecommand{\norm}[1]{\left\Vert#1\right\Vert}

\newcounter{RomanNumber}

\def\be{\begin{equation}}
\def\en{\end{equation}}
\def\bs{\begin{split}}
\def\es{\end{split}}

\allowdisplaybreaks[4]

\title[Global well--posedness and large time behavior
to a generic two---fluid model] {Global well--posedness and large
time behavior of classical solutions to a generic compressible
two--fluid model}
\author{Guochun Wu}
\address{Guochun Wu \newline Fujian Province University Key Laboratory of Computational Science, School of Mathematical Sciences, Huaqiao University, Quanzhou 362021, P.R. China.}
\email{guochunwu@126.com}
\author{Lei Yao}
\address{Lei Yao \newline School of Mathematics and Center for Nonlinear Studies, Northwest University, Xi'an 710127,P.R. China.}
\email{leiyao@nwu.edu.cn}
\author{Yinghui Zhang*}
\address{Yinghui Zhang \newline School of Mathematics and Statistics, Guangxi Normal University, Guilin, Guangxi 541004, P.R.
China} \email{yinghuizhang@mailbox.gxnu.edu.cn}

\thanks{* Corresponding author: yinghuizhang@mailbox.gxnu.edu.cn}
\bigbreak

\date{\today}
\usepackage{hyperref}

\begin{document}
\begin{abstract}
In this paper, we investigate a generic compressible two--fluid
model with common pressure ($P^+=P^-$) in $\mathbb{R}^3$. Under some
smallness assumptions, Evje--Wang--Wen [Arch Rational Mech Anal
221:1285--1316, 2016] obtained the global solution and its optimal
decay rate for the 3D compressible two--fluid model with unequal
pressures $P^+\neq P^-$. More precisely, the capillary pressure
$f(\alpha^-\rho^-)=P^+-P^-\neq 0$ is taken into account, and is
assumed to be a strictly decreasing function near the equilibrium.
As indicated by Evje--Wang--Wen, this assumption played an key role
in their analysis and appeared to have an essential stabilization
effect on the model. However, global well--posedness of the 3D
compressible two--fluid model with common pressure has been a
challenging open problem due to the fact that the system is
partially dissipative and its nonlinear structure is very terrible.
In the present work, by exploiting the dissipation structure of the
model and making full use of several key observations, we establish
global existence and large time behavior of classical solutions to
the 3D compressible two--fluid model with common pressure. One of
key observations here is that to closure the higher--order energy
estimates of non--dissipative variables (i.e, fraction densities
$\alpha_{\pm}\rho_\pm$), we will introduce the linear combination of
two velocities ($u^\pm$) :
$v=(2\mu^++\lambda^+)u^+-(2\mu^-+\lambda^-)u^-$ and explore its good
regularity, which is particularly better than ones of two velocities
themselves.

\end{abstract}

\maketitle

\section{\leftline {\bf{Introduction.}}}
\setcounter{equation}{0}
\subsection{Background and motivation}
It is well--known that in nature, most of the flows are multi--fluid
flows. Such a terminology includes the flows of non--miscible fluids
such as air and water; gas, oil and water. For the flows of miscible
fluids, they usually form a ``new" single fluid possessing its own
rheological properties. One interesting example is the stable
emulsion between oil and water which is a non--Newtonian fluid, but
oil and water are Newtonian ones.\par
 One of the classic examples of multi--fluid flows is
small amplitude waves propagating at the interface between air and
water, which is called a separated flow. In view of modeling, each
fluid obeys its own equation and couples with each other through the
free surface in this case. Here, the motion of the fluid is governed
by the pair of compressible Euler equations with free surface:
\begin{align}
\partial_{t} \rho_{i}+\nabla \cdot\left(\rho_{i} v_{i}\right) &=0, \quad i=1,2,\label{1.1} \\
\partial_{t}\left(\rho_{i} v_{i}\right)+\nabla \cdot\left(\rho_{i} v_{i} \otimes v_{i}\right)+\nabla p_i &=-g\rho_{i}
e_3\pm F_D.\label{1.2}
\end{align}
In above equations, $\rho_1$ and $v_1$ represent the density and
velocity of the upper fluid (air), and  $\rho_2$ and $v_2$ denote
the density and velocity of the lower fluid (water). $p_{i}$ denotes
the pressure. $-g\rho_{i} e_3$ is the gravitational force with the
constant $g>0$ the acceleration of gravity and $e_3$ the vertical
unit vector, and $F_D$ is the drag force. As mentioned before, the
two fluids (air and water) are separated by the unknown free surface
$z=\eta(x, y, t)$, which is advected with the fluids according to
the kinematic relation:
\begin{equation}\partial_t\eta=u_{1,z}-u_{1,x}\partial_x \eta-u_{1, y}\partial_y \eta\label{1.3}\end{equation}
on two sides of the surface $z=\eta$ and the pressure is continuous
across this surface.\par When the wave's amplitude becomes large
enough, wave breaking may happen. Then, in the region around the
interface between air and water, small droplets of liquid appear in
the gas, and bubbles of gas also appear in the liquid. These
inclusions might be quite small. Due to the appearances of collapse
and fragmentation, the topologies of the free surface become quite
complicated and a wide range of length scales are involved.
Therefore, we encounter the situation where two--fluid models become
relevant if not inevitable. The classic approach to simplify the
complexity of multi--phase flows and satisfy the engineer's need of
some modeling tools is the well--known volume--averaging method (see
\cite{Ishii1, Prosperetti} for details). Thus, by performing such a
procedure, one can derive a model without surface: a two--fluid
model. More precisely, we denote $\alpha^{\pm}$ by the volume
fraction of the liquid (water) and gas (air), respectively.
Therefore, $\alpha^++\alpha^-=1$. Applying the volume--averaging
procedure to the equations \eqref{1.1} and \eqref{1.2} leads to the
following generic compressible two--fluid model:
\begin{equation}\label{1.4}
\left\{\begin{array}{l}
\partial_{t}\left(\alpha^{\pm} \rho^{\pm}\right)+\operatorname{div}\left(\alpha^{\pm} \rho^{\pm} u^{\pm}\right)=0, \\
\partial_{t}\left(\alpha^{\pm} \rho^{\pm} u^{\pm}\right)+\operatorname{div}\left(\alpha^{\pm} \rho^{\pm} u^{\pm} \otimes u^{\pm}\right)
+\alpha^{\pm} \nabla P=-g\alpha^{\pm}\rho^{\pm} e_3\pm F_D,
\end{array}\right.
\end{equation}
where the two fluids are assumed to share the common pressure $P$.
\par
We have already discussed the case of water waves, where a separated
flow can lead to a two--fluid model from the viewpoint of practical
modeling. As mentioned before, two--fluid flows are very common in
nature, but also in various industry applications such as nuclear
power, chemical processing, oil and gas manufacturing. According to
the context, the models used for simulation may be very different.
However, averaged models share the same structure as \eqref{1.4}. By
introducing viscosity effects, one can generalize the above system
\eqref{1.4} to
\begin{equation}\label{1.5}
\left\{\begin{array}{l}
\partial_{t}\left(\alpha^{\pm} \rho^{\pm}\right)+\operatorname{div}\left(\alpha^{\pm} \rho^{\pm} u^{\pm}\right)=0, \\
\partial_{t}\left(\alpha^{\pm} \rho^{\pm} u^{\pm}\right)+\operatorname{div}\left(\alpha^{\pm} \rho^{\pm} u^{\pm} \otimes u^{\pm}\right)
+\alpha^{\pm} \nabla P=\operatorname{div}\left(\alpha^{\pm} \tau^{\pm}\right), \\
P=P^{\pm}\left(\rho^{\pm}\right)=A^{\pm}\left(\rho^{\pm}\right)^{\bar{\gamma}^{\pm}},
\end{array}\right.
\end{equation}
where $\rho^{\pm}(x, t) \geqq 0, u^{\pm}(x, t)$ and
$P^{\pm}\left(\rho^{\pm}\right)=A^{\pm}\left(\rho^{\pm}\right)^{\bar{\gamma}^{\pm}}$
denote the densities, the velocities of each phase, and the two
pressure functions, respectively. $\bar{\gamma}^{\pm} \geqq 1,
A^{\pm}>0$ are positive constants. In what follows, we set
$A^{+}=A^{-}=1$ without loss of any generality. Moreover,
$\tau^{\pm}$ are the viscous stress tensors
\begin{equation}\label{1.6}
\tau^{\pm}:=\mu^{\pm}\left(\nabla u^{\pm}+\nabla^{t}
u^{\pm}\right)+\lambda^{\pm} \operatorname{div} u^{\pm} \mathrm{Id},
\end{equation}
where the constants $\mu^{\pm}$ and $\lambda^{\pm}$ are shear and
bulk viscosity coefficients satisfying the physical condition:
$\mu^{\pm}>0$ and $2 \mu^{\pm}+3 \lambda^{\pm} \geqq 0,$ which
implies that $\mu^{\pm}+\lambda^{\pm}>0.$ For more information about
this model, we refer to \cite{Bear, Brennen1, Bresch1, Bresch2,
Evje1, Evje2, Evje3, Evje4, Evje8, Evje9, Friis1, Raja, Vasseur,
Yao2, Zhang4} and references therein. 
However, it is well--known
that as far as mathematical analysis of two--fluid model is
concerned, there are many technical challenges. Some of them
involve, for example:
\begin{itemize}
\item The two--fluid model is a partially dissipative system.
More precisely, there is no dissipation on the mass conservation
equations, whereas the momentum equations have viscosity
dissipations;

\item The corresponding linear system of the model has
zero eigenvalue, which makes mathematical analysis (well--posedness
and stability) of the model become quite difficult and complicated;

\item Transition to single--phase regions, i.e, regions where the mass
$\alpha^{+} \rho^{+}$ or $\alpha^{-} \rho^{-}$ becomes zero, may
occur when the volume fractions $\alpha^{\pm}$ or the densities
$\rho^{\pm}$ become zero;

\item The system is non--conservative, since the non--conservative terms $\alpha^{\pm} \nabla
P^{\pm}$ are involved in the momentum equations. This brings various
 mathematical difficulties for us to employ methods used
for single phase models to the two--fluid model.

\end{itemize}\par
In the excellent work \cite{Evje9}, Evje--Wang--Wen investigated the
two--fluid model \eqref{1.5} with unequal pressures. As a matter of
fact, they made the following assumptions on pressures:
\begin{equation}P^+(\rho^+)-P^-(\rho^-)=(\rho^+)^{\bar{\gamma}^+}-(\rho^-)^{\bar{\gamma}^-}=f(\alpha_-\rho_-),\label{1.7}
\end{equation}
where $f$ is so--called capillary pressure which belongs to $C^3([0,
\infty))$, and is a strictly decreasing function near the
equilibrium satisfying
\begin{equation}-\frac{s_{-}^{2}(1,1)}{\alpha^{-}(1,1)}<f^{\prime}(1)<\frac{\bar{\eta}-s_{-}^{2}(1,1)}{\alpha^{-}(1,1)}<0,\label{1.8}
\end{equation}
where $\bar{\eta}$ is a positive, small fixed constant, and
$s_{\pm}^{2}:=\frac{\mathrm{d} P^{\pm}}{\mathrm{d}
\rho^{\pm}}\left(\rho^{\pm}\right)=\bar{\gamma}^{\pm}
\frac{P^{\pm}\left(\rho^{\pm}\right)}{\rho^{\pm}}$ represent the
sound speed of each phase respectively. Under the assumptions
\eqref{1.7} and \eqref{1.8} on pressures, they obtained global
existence and decay rates of the solutions when the initial
perturbation is sufficiently small. However, as indicated by
Evje--Wang--Wen in \cite{Evje9}, assumptions \eqref{1.7} and
\eqref{1.8} played an key role in their analysis and appeared to
have an essential stabilization effect on the model in question. On
the other hand, Bresch et al. in the seminal work \cite{Bresch1}
considered a model similar to \eqref{1.5}. More specifically, they
made the following assumptions:
\begin{itemize}
\item inclusion of viscous terms of the form \eqref{1.6} where $\mu^{\pm}(\rho)=\mu_{\pm}\rho^{\pm}$
and $\lambda^{\pm}(\rho^\pm)=0$;

\item inclusion of a third order derivative of $\alpha^{\pm}
\rho^{\pm}$, which are so--called  internal capillary forces
represented by the well--known Korteweg model on each phase.
\end{itemize}
They obtained the global weak solutions in the periodic domain with
$1<\overline{\gamma}^{\pm}< 6$. It is worth mentioning that as
indicated by Bresch et al. in \cite{Bresch1}, their method cannot
handle the case without the internal capillary forces. Later,
Bresch--Huang--Li \cite{Bresch2} established the global existence of
weak solutions in one space dimension without the internal capillary
forces when $\overline{\gamma}^{\pm}>1$ by taking advantage of the
one space dimension. However, the methods of Bresch--Huang--Li
\cite{Bresch2} relied crucially on the advantage of one space
dimension, and particularly cannot be applied for high dimensional
problem. Recently, Cui--Wang--Yao--Zhu \cite{c1} obtained the
time--decay rates of classical solutions for the three--dimensional
Cauchy problem by combining detailed analysis of the Green's
function to the linearized system with delicate energy estimates to
the nonlinear system. It should be remarked that the internal
capillary forces played an essential role in the analysis of
\cite{Bresch1, c1}, which will be explained later.
 \par
  In conclusion, all the works \cite{Bresch1, c1, Evje9} depend essentially on the internal capillary
  forces effects or the capillary pressure effects. Therefore, a natural and
important problem is that what will happen if no internal capillary
force is involved and the capillary pressure $f=0$. That is to say,
what about the global well--posedness and large time behavior of
Cauchy problem to the two--fluid model \eqref{1.5} in high
dimensions. However, to our best knowledge, so far there is no
result on mathematical theory of the two--fluid model \eqref{1.5} in
high dimensions. In particular, global well--posedness of the 3D
compressible two--fluid model \eqref{1.5} has been a challenging
open problem due to the fact that the system is partially
dissipative and its nonlinear structure is very terrible. The main
purpose of this work is to resolve this problem. More precisely, by
exploiting the dissipation structure of the model and making full
use of several key observations, we establish global existence and
large time behavior of classical solutions to the 3D compressible
two--fluid model \eqref{1.5}. One of key observations here is that
to closure the higher--order energy estimates of non--dissipative
variables (i.e, fraction densities $\alpha_{\pm}\rho_\pm$), we will
introduce the linear combination of two velocities ($u^\pm$) :
$v=(2\mu^++\lambda^+)u^+-(2\mu^-+\lambda^-)u^-$ and explore its good
regularity, which is particularly better than ones of two velocities
themselves. Particularly, our results show that even if both the
capillary pressure effects and internal capillary forces effects are
absence, viscosity forces still may prevent the formation of
singularities for the model in question. Furthermore, the components
of the solution exhibit totally distinctive behaviors: the
dissipation variables have decay rates in time, but the
non--dissipation variables only have uniform time--independent
bounds. This phenomenon is totally new as compared to \cite{c1,
Evje9} where all the components of the solution show the same
behaviors (i.e, have decay rates in time), and is the most important
difference between partially dissipative system and dissipative
system.

\subsection{New formulation of system \eqref{1.5} and Main Results}
In this subsection, we devote ourselves to reformulating the system
\eqref{1.5} and stating the main results. To begin with, noticing
the relation between the pressures of \eqref{1.5}$_3$, one has
\begin{equation}\label{1.9}
\mathrm{d} P=s_{+}^{2} \mathrm{d} \rho^{+} =s_{-}^{2} \mathrm{d} \rho^{-},
\end{equation}
 where $s_{\pm}^{2}:=\frac{\mathrm{d} P^{\pm}}{\mathrm{d}
\rho^{\pm}}\left(\rho^{\pm}\right)=\bar{\gamma}^{\pm}
\frac{P^{\pm}\left(\rho^{\pm}\right)}{\rho^{\pm}}$ represent the
sound speed of each phase respectively. As in \cite{Bresch1}, we
introduce the fraction densities
\begin{equation}\label{1.10}
R^{\pm}=\alpha^{\pm} \rho^{\pm},
\end{equation}
which together with the relation: $\alpha^++\alpha^-=1$ gives
\begin{equation}\label{1.11}
\mathrm{d} \rho^{+}=\frac{1}{\alpha_{+}}\left(\mathrm{d}
R^{+}-\rho^{+} \mathrm{d} \alpha^{+}\right), \quad \mathrm{d}
\rho^{-}=\frac{1}{\alpha_{-}}\left(\mathrm{d} R^{-}+\rho^{-}
\mathrm{d} \alpha^{+}\right).
\end{equation}
By virtue of \eqref{1.9} and \eqref{1.10}, we finally get
\begin{equation}\label{1.12}
\mathrm{d} \alpha^{+}=\frac{\alpha^{-} s_{+}^{2}}{\alpha^{-}
\rho^{+} s_{+}^{2}+\alpha^{+} \rho^{-} s_{-}^{2}} \mathrm{d}
R^{+}-\frac{\alpha^{+} s_{-}^{2}}{\alpha^{-} \rho^{+}
s_{+}^{2}+\alpha^{+} \rho^{-}
s_{-}^{2}}
\mathrm{d} R^{-}. \end{equation}
 Substituting \eqref{1.12} into \eqref{1.11}, we deduce
the following expressions:
\[
\mathrm{d} \rho^{+}=\frac{\rho^{+} \rho^{-}
s_{-}^{2}}{R^{-}\left(\rho^{+}\right)^{2}
s_{+}^{2}+R^{+}\left(\rho^{-}\right)^{2} s_{-}^{2}}\left(\rho^{-}
\mathrm{d} R^{+}+\rho^{+}\mathrm{d} R^{-}\right),
\]
and
\[
\mathrm{d} \rho^{-}=\frac{\rho^{+} \rho^{-}
s_{+}^{2}}{R^{-}\left(\rho^{+}\right)^{2}
s_{+}^{2}+R^{+}\left(\rho^{-}\right)^{2} s_{-}^{2}}\left(\rho^{-}
\mathrm{d} R^{+}+\rho^{+} \mathrm{d} R^{-}\right),
\]
which together with \eqref{1.9} give the common pressure
differential $\mathrm{d} P$
\[
\mathrm{d} P=\mathcal{C}\left(\rho^{-} \mathrm{d}
R^{+}+\rho^{+} \mathrm{d} R^{-}\right) ,\]  where
\[
\mathcal{C}:=\frac{s_{-}^{2} s_{+}^{2}}{\alpha^{-} \rho^{+}
s_{+}^{2}+\alpha^{+} \rho^{-} s_{-}^{2}}.\]

Next, by noting the fundamental relation: $\alpha^++\alpha^-=1$, we
can get the following equality:
\begin{equation}\label{1.13}
\frac{R^+}{\rho^+}+\frac{R^-}{\rho^-}=1, ~~\hbox{and thus}~~
R^-=\frac{\rho^-\left(\rho^+-R^+\right)}{\rho^+}=\frac{P^{1/\gamma^-}\left(P^{1/\gamma^+}-R^+\right)}{P^{1/\gamma^+}}.\end{equation}

\noindent  By virtue of $\eqref{1.5}_3$, \eqref{1.10} and
\eqref{1.13}, $\alpha^{\pm}$ can be defined by
\begin{equation}\label{1.14}
\left\{\begin{array}{l}
\alpha^{+}\left(P,R^{+}\right) =\frac{R^{+}}{P^{1/\gamma^+}}, \\
\alpha^{-}\left(P,R^{+}\right)
=1-\frac{R^{+}}{P^{1/\gamma^+}}.
\end{array}\right.
\end{equation}
We refer the readers to [\cite{Bresch1}, P. 614] for more details.
\par
As already stated, the model \eqref{1.5} is a partially dissipative
system, which brings various difficulties for our studies on its
mathematical properties. Therefore, to tackle with this difficulty,
we need to explore new dissipative variable by making full use of
the structure of the model \eqref{1.5}. One key observation in this
paper is that the common pressure $P$ is a dissipative variable.
With this crucial observation, the system \eqref{1.5} can be
rewritten in terms of the variables $(R^+, P, u^+, u^- )$:
\begin{equation}\label{1.15}
\left\{\begin{array}{l}
\partial_{t} R^{+}+\operatorname{div}\left(R^{+} u^{+}\right)=0, \\
\partial_{t} P+\mathcal C\rho^-\operatorname{div}\left(R^{+} u^{+}\right)+\mathcal C\rho^+\operatorname{div}\left(R^{-} u^{-}\right)=0,\\
\partial_{t}\left(R^{+} u^{+}\right)+\operatorname{div}\left(R^{+} u^{+} \otimes u^{+}\right)+\alpha^{+} \nabla P \\
\hspace{2.5cm}=\operatorname{div}\left\{\alpha^{+}\left[\mu^{+}\left(\nabla u^{+}+\nabla^{t} u^{+}\right)
+\lambda^{+} \operatorname{div} u^{+} \operatorname{Id}\right]\right\}, \\
\partial_{t}\left(R^{-} u^{-}\right)+\operatorname{div}\left(R^{-} u^{-} \otimes u^{-}\right)+\alpha^{-} \nabla P \\
\hspace{2.5cm}=\operatorname{div}\left\{\alpha^{-}\left[\mu^{-}\left(\nabla
u^{-}+\nabla^{t} u^{-}\right)+\lambda^{-} \operatorname{div} u^{-}
\operatorname{Id}\right]\right\}.
\end{array}\right.
\end{equation}
In the present paper, we consider the initial value problem to
\eqref{1.15} in the whole space $\mathbb R^3$ with the initial data
\begin{equation}\label{1.16} (R^{+}, P, u^{+}, u^{-})(x,
0)=(R_{0}^{+}, P_{0}, u_{0}^{+},  u_{0}^{-})(x)\rightarrow(\bar
R^{+}, \bar P, \overrightarrow{0},  \overrightarrow{0}) \quad
\hbox{as}\quad |x|\rightarrow\infty \in \mathbb{R}^{3},
\end{equation}
where two positive constants  ${\bar{R}}^{+}$ and $\bar{P}$
represent the background doping profile, and for simplicity, we take
${\bar{R}}^{\pm}$ as 1, and thus $\bar{P}$ is determined by the
relation: $R^-(\bar{P}, 1)-1=0$.
\medskip
\par
Now, we are in a position to state our main result.
\smallskip
\begin{Theorem}\label{1mainth}Assume that $R_{0}^{+}-1\in H^3(\mathbb{R}^3),P_0-\bar{P}, u_{0}^{+},
u_{0}^{-}\in H^3(\mathbb{R}^3)\cap L^1(\mathbb{R}^3)$, then
there exists a constant $\delta_0$ such that if
\begin{equation}\label{1.17}K_0:=\left\|\left(R_{0}^{+}-1\right)\right\|_{H^3}+
\left\|\left( P_0-\bar{P},u_{0}^{+},
u_{0}^{-}\right)\right\|_{H^{3}\cap L^1} \leq \delta_0,
\end{equation}
then the Cauchy problem \eqref{1.15}--\eqref{1.16} admits a unique
solution $\left(R^{+}, P,u^{+}, u^{-}\right)$ globally in time,
satisfying
\[
\begin{array}{l}
R^{+}-1, P-\bar{P} \in C^{0}\left([0, \infty) ;
H^{3}\left(\mathbb{R}^{3}\right)\right) \cap C^{1}\left([0, \infty)
;
H^{2}\left(\mathbb{R}^{3}\right)\right), \\
u^{+}, u^{-} \in C^{0}\left([0, \infty) ;
H^{3}\left(\mathbb{R}^{3}\right)\right) \cap C^{1}\left([0, \infty)
; H^1\left(\mathbb{R}^{3}\right)\right).
\end{array}
\]
Moreover, for any $t\ge 0$, there exists a positive constant $C_0$
independent of $t$ such that the solution $\left(R^{+}, P,u^{+},
u^{-}\right)$ satisfies the following estimates:
\begin{equation}\label{1.18}\left\|\left( P-\bar{P},u^{+},
u^{-}\right)(t)\right\|_{L^2}\le C_0K_0(1+t)^{-\frac{3}{4}},
\end{equation}
\begin{equation}\label{1.19}\left\|\nabla P(t)\right\|_{H^1}+\left\|\nabla\left( u^{+},
u^{-}\right)(t)\right\|_{H^2}\le C_0K_0(1+t)^{-\frac{5}{4}}.
\end{equation}
and
\begin{equation}\label{1.20}\|\nabla^3P(t)\|_{L^2}+\|(R^+-1,R^--1)(t)\|_{H^3}\le C_0K_0.
\end{equation}
\end{Theorem}

\begin{remark} \eqref{1.18}--\eqref{1.20} show that the components of the solution exhibit
totally distinctive behaviors. More precisely, the dissipation
variables $(P, u^+, u^-)$ have decay rates in time, but the
non--dissipation variables $(R^+, R^-)$ only have uniform
time--independent bound. This phenomenon is totally new as compared
to \cite{c1, Evje9} where all the components of the solution show
the same behaviors (i.e, have decay rates in time), and is the most
important difference between partially dissipative system and
dissipative system.
\end{remark}
\begin{remark} Noticing the lower bound on linear $L^2$ decay rates in Proposition \ref{Prop2.6}, we can
employ Duhamel principle and uniform time--independent energy
estimates on the nonlinear terms to prove $\left\|\left(
P-\bar{P},u^{+}, u^{-}\right)(t)\right\|_{L^2}\geq
C(1+t)^{-\frac{3}{4}}$, where the rate is the same as the upper
decay rate in \eqref{1.18}. Therefore, our $L^2$ decay rate on $(
P-\bar{P},u^{+}, u^{-})$ is optimal in this sense. However, this is
not our main concern. We will focus our attention on global
well--posedness of the solutions and thus omit the details for the
sake of simplicity.
\end{remark}

\begin{remark} Our methods can be applied to study bounded domain problem to the 3D compressible
two--fluid model \eqref{1.5}. This will be reported in our
forthcoming work \cite{WYZ}.
\end{remark}

\smallskip
\smallskip

\indent Now, let us illustrate the main difficulties encountered in
proving Theorem \ref{1mainth} and explain our strategies to overcome
them. Different from the models in \cite{Evje9, c1}, where either
the capillary pressure effects or the internal capillary forces
effects are involved, we consider the two--fluid model \eqref{1.5}
with common pressure $P^+=P^-$ and no capillary forces effects being
involved. Therefore, the model in the present paper is a partially
dissipative system which brings many essential difficulties. We will
develop new ideas to overcome these difficulties as explained
below.\par
 To begin with, we give a heuristic description of the significant difference between model \eqref{1.5} and those in  \cite{Evje9, c1}.
By taking $n^{\pm}=R^{\pm}-1$, one can write the corresponding
linear system of the model \eqref{1.5} in terms of the variables
$(n^+, u^+, n^-, u^-)$:
\begin{equation}\label{1.21}\left\{\begin{array}{l}
\partial_{t} n^{+}+\operatorname{div} u^{+}=0, \\
\partial_{t} u^{+}+\bar{\beta}_{1} \nabla n^{+}+\bar{\beta}_{2} \nabla n^{-}-\nu_{1}^{+} \Delta u^{+}-
\nu_{2}^{+} \nabla \operatorname{div} u^{+}=0, \\
\partial_{t} n^{-}+\operatorname{div} u^{-}=0, \\
\partial_{t} u^{-}+\bar{\beta}_{3} \nabla n^{+}+\bar{\beta}_{4} \nabla n^{-}-\nu_{1}^{-} \Delta u^{-}-
\nu_{2}^{-} \nabla \operatorname{div} u^{-}=0,
\end{array}\right.\end{equation}
where $\bar{\beta}_{1}=\frac{\mathcal{C}(1,1)
\rho^{-}(1,1)}{\rho^{+}(1,1)}$,
$\bar{\beta}_{2}=\bar{\beta}_{3}=\mathcal{C}(1,1)$,
$\bar{\beta}_{4}=\frac{\mathcal{C}(1,1)
\rho^{+}(1,1)}{\rho^{-}(1,1)}$,
$\nu_{1}^{\pm}=\frac{\mu^{\pm}}{\rho^{\pm}(1,1)}$, and
$\nu_{2}^{\pm}=\frac{\mu^{\pm}+\lambda^{\pm}}{\rho^{\pm}(1,1)}>0$.
Multiplying $\eqref{1.21}_1$, $\eqref{1.21}_2$, $\eqref{1.21}_3$ and
$\eqref{1.21}_4$ by $\frac{\bar{\beta}_1}{\bar{\beta}_2}n^+$,
$\frac{1}{\bar{\beta}_2}u^+$,
$\frac{\bar{\beta}_4}{\bar{\beta}_3}n^-$ and
$\frac{1}{\bar{\beta}_3}u^-$, one can easily get the nature energy
equation of the linear system \eqref{1.21}:
\begin{align}\begin{split}\label{1.22}
\partial_t\mathcal{E}_0(t)+\mathcal{D}_0(t)&:=\displaystyle\partial_t\int_{\mathbb{R}^3}\left(
\frac{\bar{\beta}_1}{2\bar{\beta}_2}\left|n^+\right|^2+\frac{\bar{\beta}_4}{2\bar{\beta}_3}\left|n^-\right|^2+n^+n^-+
\frac{1}{2\bar{\beta}_2}\left|u^+\right|^2+\frac{1}{2\bar{\beta}_3}\left|u^-\right|^2\right)\textrm{d}x\\
&\quad+\displaystyle\int_{\mathbb{R}^3}\frac{1}{\bar{\beta}_2}\left(
\nu_1^+\left|\nabla
u^+\right|^2+\nu_2^+\left|\hbox{div}u^+\right|^2\right)+\frac{1}{\bar{\beta}_3}\left(
\nu_1^-\left|\nabla
u^-\right|^2+\nu_2^-\left|\hbox{div}u^-\right|^2\right)\textrm{d}x=0,
\end{split}\end{align}
where $\mathcal{E}_0(t)$ and $\mathcal{D}_0(t)$  denote the nature
energy and dissipation, respectively. Noticing the fact that
$\bar{\beta}_1\bar{\beta}_4=\bar{\beta}_2\bar{\beta}_3=\bar{\beta}_3^2$,
it is clear that
\begin{equation}\displaystyle\mathcal{E}_0(t)=\frac{1}{2}\int_{\mathbb{R}^3}\left(
\left(\sqrt{\frac{\bar{\beta}_1}{\bar{\beta}_2}}n^++\sqrt{\frac{\bar{\beta}_4}{\bar{\beta}_3}}n^-\right)^2
+\frac{1}{\bar{\beta}_2}\left|u^+\right|^2+\frac{1}{\bar{\beta}_3}\left|u^-\right|^2\right)\textrm{d}x.\nonumber\end{equation}
This together with the energy equation \eqref{1.22} makes it
impossible for us to get the uniform energy estimates of the
non--dissipative variables $n^+$ and $n^-$ simultaneously, even
though in the linear level, but possibly the uniform energy estimate
of their linear combination
$\sqrt{\frac{\bar{\beta}_1}{\bar{\beta}_2}}n^++\sqrt{\frac{\bar{\beta}_4}{\bar{\beta}_3}}n^-$.
On the other hand, Evje--Wang--Wen \cite{Evje9} considered the
capillary pressure $f=P^+-P^-\neq 0$ which is a strictly decreasing
function near the equilibrium and satisfies assumption \eqref{1.8}.
Similarly, we can get the same nature energy equation \eqref{1.22}
except replacing the expressions of $\bar{\beta}_2$ and
$\bar{\beta}_4$ by
$\bar{\beta}_{2}=\mathcal{C}(1,1)+\frac{\mathcal{C}(1,1)
\alpha^-(1,1)f'(1)}{{s^2_{-}}(1,1)}$ and
$\bar{\beta}_{4}=\frac{\mathcal{C}(1,1)
\rho^{+}(1,1)}{\rho^{-}(1,1)}-\frac{\mathcal{C}(1,1)
\alpha^+(1,1)f'(1)}{{s^2_{+}}(1,1)}$. Then, we can rewrite the
energy $\mathcal{E}_0(t)$ into:
\begin{align*}\mathcal{E}_0(t)&=\displaystyle\frac{1}{2}\int_{\mathbb{R}^3}\left(
\left(\sqrt{\frac{\bar{\beta}_1}{\bar{\beta}_2}}n^++\sqrt{\frac{\bar{\beta}_2}{\bar{\beta}_1}}n^-\right)^2
+\left(\frac{\bar{\beta}_4}{\bar{\beta}_3}-\frac{\bar{\beta}_2}{\bar{\beta}_1}\right)\left|n^-\right|^2
+\frac{1}{\bar{\beta}_2}\left|u^+\right|^2+\frac{1}{\bar{\beta}_3}\left|u^-\right|^2\right)\mathrm{d}x\\
&=\displaystyle\frac{1}{2}\int_{\mathbb{R}^3}\left(
\left(\sqrt{\frac{\bar{\beta}_1}{\bar{\beta}_2}}n^++\sqrt{\frac{\bar{\beta}_2}{\bar{\beta}_1}}n^-\right)^2
-\frac{\mathcal{C}^2(1,1)f'(1)}{\bar{\beta}_1\bar{\beta}_3\rho^+(1)}\left|n^-\right|^2
+\frac{1}{\bar{\beta}_2}\left|u^+\right|^2+\frac{1}{\bar{\beta}_3}\left|u^-\right|^2\right)\mathrm{d}x,
\end{align*}
which together with the key assumption \eqref{1.8} implies that
Evje--Wang--Wen \cite{Evje9} can get the uniform energy estimates of
$n^+$ and $n^-$ simultaneously, at least in the linear level. As for
Cui--Wang--Yao--Zhu \cite{c1} where the internal capillarity effects
in terms of a third--order derivative of fraction density
($\alpha^\pm\rho^\pm$) are involved in the momentum equations, one
can get the similar nature energy equation \eqref{1.22} except
adding the term
$\displaystyle\int_{\mathbb{R}^3}\frac{\sigma^+}{\bar{\beta}_2}|\nabla
n^+|^2+\frac{\sigma^-}{\bar{\beta}_3}|\nabla n^-|^2\mathrm{d}x$ into
the energy $\mathcal{E}_0(t)$ with the capillary coefficients
$\sigma^{\pm}>0$. This again enables Cui--Wang--Yao--Zhu \cite{c1}
to obtain the uniform energy estimates of $n^+$ and $n^-$, at least
in the linear level. Therefore, the methods of \cite{Evje9, c1}
relying heavily
 on the capillary pressure effects or the internal capillary forces
effects are invalid in our problem. As already stated, for our
studies on the partially dissipative system \eqref{1.5}, it is
essential to explore some new potential dissipative variable by
fully using the specific structure of the system, and divide the
analysis into two parts, one for the dissipative variables and
another one for the non--dissipative variables. As mentioned above,
the linear combination:
$\displaystyle\sqrt{\frac{\bar{\beta}_1}{\bar{\beta}_2}}n^++\sqrt{\frac{\bar{\beta}_4}{\bar{\beta}_3}}n^-=
\frac{1}{\sqrt{\rho^-(1,1)\rho^+({1,1})}}(\rho^-(1,1)n^++\rho^+(1,1)n^-)$
may be dissipative. On the other hand, by virtue of Mean Value
Theorem, we have $P-\bar{P}\sim \mathcal{C}(1,1)\left(\rho^-(1,1)
n^++\rho^+(1,1) n^-\right)$. In the spirit of these heuristic
observations, it is nature to choose the common pressure $P$ as a
new variable, and thus, in terms of the variables $(R^+, P, u^+,
u^-)$, we can reformulate the Cauchy problem of the model \eqref
{1.5} into the Cauchy problem \eqref{1.15}--\eqref{1.16}. Then, we
first do energy estimates on the dissipative variables. Next, we
deduce energy estimates on the non--dissipative variables to close
the a priori assumption \eqref{3.1} by making full use of the
obtained energy estimates on the dissipative variables. Roughly
speaking, our proofs mainly involves the following three steps.
\par First, we make spectral analysis and linear $L^2$ estimates on
dissipative components $(\theta, u^+, u^-)$ of the solution to the
linear system of \eqref{2.1}--\eqref{2.2}. To derive time--decay
estimates of the linear system \eqref{2.12}, it requires us to make
a detailed analysis on the properties of the semigroup. To simplify
the analysis of the Green function which is a $7\times 7$ system, we
will employ the Hodge decomposition technique firstly introduced by
Danchin \cite{Dan1} to split the linear system into three systems.
One has three distinct eigenvalues and the other two are classic
heat equations. By making careful pointwise estimates on the Fourier
transform of Green's function to the linear equations, we can obtain
the desired linear $L^2$ estimates, and refer to the proof of
Proposition \ref{2.6} for details.
\par Second, we deduce decay estimates of $(\theta, u^+, u^-)$.
Under a priori assumption: $\|(n^{\pm}, \theta, u^{\pm})\|_{H^3}$
$\leq \delta\ll 1$, we first deduce $H^2$--energy estimates of
$(\theta, u^+, u^-)$. Here, it is worth mentioning that at this
stage, we cannot derive $H^3$--energy estimates as the classic
results for the dissipation system. Indeed, since the common
pressure $P$ has been chosen as a new variable in our analysis, the
cost is that the strongly coupling terms like $\mathcal
C\rho^-u^+\cdot\nabla n^+$ and $\mathcal C\rho^+u^-\cdot\nabla n^- $
are involved in the right--hand side of equation $\eqref{2.1}_2$.
Therefore, in the derivation of the a priori energy estimate for
$\nabla^3\theta$, we encounter the trouble term likes
$\displaystyle\int_{\mathbb{R}^3}\mathcal{C}\rho^{\mp}u^\pm\cdot\nabla\nabla^3n^{\pm}\nabla^3\theta
\mathrm{d}x $, which however is out of control in the $H^3$ setting.
By making careful analysis and interpolation tricks, we can get
\begin{equation}
\displaystyle \frac{d}{dt} \mathcal{E}_1(t)+C\left(\|\nabla^2
u^\pm\|_{H^1}^2+\|\nabla\hbox{div}u^\pm\|_{H^1}^2\right)\lesssim
\delta\left(\|\nabla \theta\|_{H^1}^2+\|\nabla
u^\pm\|_{H^2}^2\right),\nonumber
\end{equation}
where $\mathcal{E}_1(t)$ is equivalent to $\|\nabla(\theta, u^+,
u^-)\|_{H^1}^2$. Next, by making energy estimates on the time
derivatives of $(\theta, u^+, u^-)$ and fully using parabolic
properties of the momentum equations in \eqref{2.1}, we deduce that
\begin{equation}
\displaystyle \frac{d}{dt}
\mathcal{E}_2(t)+C\left(\mathcal{E}_2(t)+\|u^{\pm}_{tt}\|_{L^2}^2\right)\lesssim
\delta\|\nabla (\theta^l, u^{\pm, l})\|_{L^2}^2,\nonumber
\end{equation}
where $\mathcal{E}_2(t)=\mathcal{E}_1(t)+\|(\theta_t, u^{\pm }_t,
\nabla u^{\pm}_t\|_{L^2}^2$.  Therefore, to close the energy
estimate of $\mathcal{E}_2(t)$, it suffices to show that $\|\nabla
(\theta^l, u^{\pm, l})\|_{L^2}$ decays sufficiently quickly.
Applying Duhamel principle, linear estimates obtained in Step 1,
Plancherel theorem, H\"{o}lder inequality and Hausdorff--Young
inequality, one has
\begin{equation}\label{1.23}\|\nabla(\theta^l,u^{\pm,l})(t)\|_{L^{2}}\lesssim K_0(1+t)^{-\frac{5}{4}}
+\int_0^t\left\|\nabla \text{e}^{(t-\tau)\mathcal{B}}\mathcal{
F}^l(\tau)\right\|_{L^2}\mathrm{d}\tau,
\end{equation}
where $\mathcal{F}=(F^1, F^2, F^3)^t$ denote the nonlinear source
terms. On the other hand, since $n^\pm$ are non--dissipative, the
strongly coupling terms $\mathcal C\rho^-u^+\cdot\nabla n^+$ and
$\mathcal C\rho^+u^-\cdot\nabla n^-$ in \eqref{2.3} devote the
slowest time--decay rates into the second term on the right--hand of
\eqref{1.23}. Therefore, this together with \eqref{1.23} implies
that one can only get the following estimate:
\begin{align*}&\|\nabla(\theta^l,u^{\pm,l})(t)\|_{L^{2}}\\
&\quad\lesssim K_0(1+t)^{-\frac{5}{4}}
+\delta_0\int_0^{\frac{t}{2}}\left\|\nabla
\text{e}^{(t-\tau)\mathcal{B}}\mathcal{
F}^l(\tau)\right\|_{L^2}\mathrm{d}\tau+\delta_0\int^t_{\frac{t}{2}}\left\|\nabla
\text{e}^{(t-\tau)\mathcal{B}}\mathcal{
F}^l(\tau)\right\|_{L^2}\mathrm{d}\tau\\
&\quad\lesssim K_0(1+t)^{-\frac{5}{4}}
+\delta_0\int_0^{\frac{t}{2}}(1+t-\tau)^{-\frac{5}{4}}\left\|\mathcal{F}^l(\tau)\right\|_{L^1}\mathrm{d}\tau
+\delta_0\int_{\frac{t}{2}}^t(1+t-\tau)^{-\frac{1}{2}}\left\|\mathcal{F}^l(\tau)\right\|_{L^2}\mathrm{d}\tau\\
&\quad\lesssim K_0(1+t)^{-\frac{5}{4}}
+\delta_0\int_0^{\frac{t}{2}}(1+t-\tau)^{-\frac{5}{4}}(1+\tau)^{-\frac{3}{4}}\mathrm{d}\tau
+\delta_0\int_{\frac{t}{2}}^t(1+t-\tau)^{-\frac{1}{2}}(1+\tau)^{-\frac{3}{2}}\mathrm{d}\tau\\
&\quad\lesssim K_0(1+t)^{-\frac{5}{4}}+\delta_0(1+t)^{-1},
\end{align*}
which however is not quickly enough for us to close energy estimates
of the non--dissipative variables $n^\pm$ (see the proof of Lemma
\ref{Lemma3.5} for details). To overcome this difficulty, it is
essential to develop new ideas to deal with the trouble terms
$\mathcal C\rho^-u^+\cdot\nabla n^+$ and $\mathcal
C\rho^+u^-\cdot\nabla n^-$. The key idea here is that we consider
the two trouble terms as one: $\mathcal C\rho^-u^+\cdot\nabla
n^++\mathcal C\rho^+u^-\cdot\nabla n^-$ and rewrite it in a clever
way. More specifically, by noticing that
$\rho^-s^2_--\rho^+s^2_+\neq 0$ and fully using the subtle relation
between the variables, we surprisingly find that
\begin{equation}\begin{split}
&\mathcal C\rho^-u^+\cdot\nabla n^++\mathcal C\rho^+u^-\cdot\nabla
n^-\\
&=\text{div}\left[\frac{s_+^2s_-^2\rho^+\rho^-(u^+-u^-)}{\rho^-s_-^2-\rho^+s_+^2}
\ln\left(\frac{
(\rho^+s_+)^2+(\rho^-s_-^2-\rho^+s_+^2)R^+}{(\bar\rho^+\bar
s_+)^2+(\bar\rho^-\bar s_-^2-\bar\rho^+\bar s_+^2)\bar R^+}\right)
\right]+\hbox{good terms}\\
&:=\hbox{div}\mathcal{R}_1+\hbox{good terms}.\label{1.24}
\end{split}\end{equation}
With this crucial observation, one can shift the derivative of
$\mathcal{R}_1$ onto the solution semigroup to derive the desired
decay estimates of $\|\nabla (\theta^l, u^{\pm, l})\|_{L^2}$, and
refer to the proof of \eqref{3.31} for details. Consequently, by
noticing the definition of $\mathcal{E}_2(t)$ and using the
parabolicity of the momentum equations in \eqref{2.1}, we can obtain
the key uniform estimate on $\|\nabla \theta\|_{H^1}+\|\nabla
u^{\pm}\|_{H^2}$. Finally, by using this crucial uniform estimate
and employing similar arguments, we can get the desired uniform
estimate of $\|(\theta, u^+, u^-)\|_{L^2}$.
\par In the last step, we close energy estimates of $n^\pm$.
Compared to \cite{c1, Evje9}, we need to develop new ingredients in
the proof to handle with the difficulties arising from the partially
dissipation system, which requires some new thoughts. Indeed, as in
\cite{c1, Evje9}, one may consider the corresponding linearized
system of the model \eqref{1.5}:
\begin{equation}\label{1.25}\left\{\begin{array}{l}
\partial_{t} n^{+}+\operatorname{div} u^{+}=S_1^+, \\
\partial_{t} u^{+}+\bar{\beta}_{1} \nabla n^{+}+\bar{\beta}_{2} \nabla n^{-}-\nu_{1}^{+} \Delta u^{+}-
\nu_{2}^{+} \nabla \operatorname{div} u^{+}=S_2^+, \\
\partial_{t} n^{-}+\operatorname{div} u^{-}=S_1^-, \\
\partial_{t} u^{-}+\bar{\beta}_{3} \nabla n^{+}+\bar{\beta}_{4} \nabla n^{-}-\nu_{1}^{-} \Delta u^{-}-
\nu_{2}^{-} \nabla \operatorname{div} u^{-}=S_2^-,
\end{array}\right.\end{equation}
where $S_1^{\pm}=-\hbox{div}(n^{\pm}u^\pm)$, and
\begin{align}\begin{split}S_2^{\pm}&=-\left(\frac{\mathcal{C}^2\rho^{\mp}(n^++1, \theta+\bar{P})}{\rho^{\pm}(n^++1,
\theta+\bar{P})}-\frac{\mathcal{C}^2\rho^{\mp}(1,
\bar{P})}{\rho^{\pm}(1,
\bar{P})}\right)\nabla n^{\pm}\\
&\quad-\left(\mathcal{C}^2(n^++1, \theta+\bar{P})-\mathcal{C}^2(1,
\bar{P})\right)\nabla n^{\mp}+\hbox{good
terms}.\label{1.26}\end{split}\end{align} Then, by virtue of the
first and third equations in \eqref{1.25}, it is easily to see that
\begin{equation}\label{1.27}
\frac{d}{dt}\|n^{\pm}\|_{H^k}\lesssim\|\nabla u^{\pm}\|_{H^k},~~
\hbox{for}~~ k=2,3,
\end{equation}
which together with the key uniform decay estimates of $\|\nabla
u^{\pm}\|_{H^2}$ in \eqref{3.22} implies the uniform boundedness of
$\|n^{\pm}\|_{H^2}$ directly. However, since we don't know whether
the $L^1(0,t)$--norm of $\|\nabla^{4}u^\pm\|_{L^2}$ is uniformly
bounded or not, it seems impossible to derive the uniform estimates
on $\|\nabla^3 n^{\pm}\|_{L^2}$ from \eqref{1.27}. Therefore, we
must pursue another route by resorting to the momentum equations in
\eqref{1.25}. By multiplying $\nabla^3 \eqref{1.25}_2$ and $\nabla^3
\eqref{1.25}_4$ by $\nabla^3 u^+$ and $\nabla^3 u^+$, and summing
up, one has from \eqref{1.26}
\begin{align}\begin{split}\label{1.28}
\displaystyle
\frac{d}{dt}\frac{1}{2}\int_{\mathbb{R}^3}\left|\nabla^3u^{\pm}\right|^2&+
\bar{\beta}_1\left|\nabla^3n^{+}\right|^2+\bar{\beta}_3\left|\nabla^3n^{-}\right|^2\mathrm{d}x\\
&+\int_{\mathbb{R}^3}\nu^{\pm}_1\left|\nabla^4u^{\pm}\right|^2+
\nu^{\pm}_2\left|\nabla^3\hbox{div}u^{\pm}\right|^2\mathrm{d}x\\
&
\hspace{-0.3cm}\begin{matrix}\approx\underbrace{\int_{\mathbb{R}^3}\nabla^3
u^{\pm}\cdot \nabla^3\left(n^\pm \nabla
n^\pm\right)\mathrm{d}x}+\hbox{other terms}.\\
\hspace{-1.75cm}\mathcal{R}_2\end{matrix}
\end{split}\end{align}
Noticing that $n^{\pm}$ has no decay rate in time, the term
$\mathcal{R}_2$ in the right--hand of \eqref{1.28} is out of control
since we don't know whether the $L^1(0,t)$--norm of
$\|\nabla^{4}u^\pm\|_{L^2}$ is uniformly bounded or not. To tackle
with this difficulty, our idea here is to introduce a new
semi--linearized system of the model \eqref{1.5} (refer to
\eqref{3.36}), where the term $\mathcal{R}_2$ exactly disappears.
However, we encounter an alternative trouble term coming from strong
couplings between the two fluids (refer to \eqref{3.45}):
\begin{equation}\displaystyle\mathcal{R}_3=\int_{\mathbb{R}^3}\nabla^3\left[\mathcal
C\left(\rho^-\nabla n^++\rho^+\nabla n^-
\right)\right]\cdot\nabla^3u^\pm \mathrm{d}x.\nonumber\end{equation}
By virtue of $\eqref{1.25}_1$ and $\eqref{1.25}_3$, we can rewrite
$\mathcal{{R}}_3$ as follows:
\begin{equation}
\begin{matrix}\hspace{0.55cm}\displaystyle\mathcal{R}_3=-\underbrace{\int_{\mathbb{R}^3}\mathcal{C}\rho^{-} \nabla^3 n^{+}\nabla^3
\text{div}u^\pm \mathrm{d}x}-\\
\hspace{1.4cm}\mathcal{R}_{3,1}\end{matrix}
\begin{matrix}\displaystyle\underbrace{\int_{\mathbb{R}^3}\mathcal{C}\rho^{+}
\nabla^3 n^{-}\nabla^3
\text{div}u^\pm \mathrm{d}x}+~\hbox{lower order terms}.\\
\hspace{-3.1cm}\mathcal{R}_{3,2}\end{matrix}\nonumber
\end{equation}
For $\mathcal{R}_{3,1}$, we can employ the equation $\eqref{1.25}_1$
to deal with. However, to bound $\mathcal{R}_{3,2}$, because of the
strongly coupling of the system and non--dissipation of $n^-$, it
requires us to develop new thoughts. The key idea here is to fully
use the special structure of the equation \eqref{3.36} and the fact
that the pressures of the two fluids are equal. To see this, we
first introduce the linear combination of two velocities $u^\pm$ as
follows:
\begin{equation}\label{1.29}v=(2\mu^++\lambda^+)u^+-(2\mu^-+\lambda^-)u^-,\end{equation}
which satisfies the following elliptic equation
\begin{equation}\label{1.30}\begin{split}-\Delta \hbox{div}v=-\text{div}(\rho^+u^+_t-\rho^-u^-_t)+\text{div}(G_1-G_2),\end{split}
\end{equation}
where $G_1$ and $G_2$ are defined in \eqref{3.36}. Then by employing
the classic elliptic estimate for \eqref{1.30}, we have (refer to
\eqref{3.41}):
\begin{equation}\nonumber\begin{split}\|\nabla^3\hbox{div}
v\|_{L^2}&\lesssim\left|\left|\nabla
u^\pm_t\right|\right|_{H^1}+\left|\left|\nabla
u^\pm\right|\right|_{H^2}\\
&\lesssim\left|\left|\left(u^\pm_{tt}, u^\pm_t, \nabla
u^\pm_t\right)\right|\right|_{L^2}+\left|\left|\nabla
u^\pm\right|\right|_{H^2},
\end{split}
\end{equation}
which together with the uniform decay estimate on
$\left|\left|u^\pm_t\right|\right|_{H^1}+\left|\left|\nabla
u^\pm\right|\right|_{H^2}$ and time--weighted estimate on
$\left|\left|u^\pm_{tt}\right|\right|_{L^2}$ in Step 2(see also \eqref{3.22} and \eqref{3.23}) leads to the
following key estimate:
\begin{equation}\label{1.31}\int_0^t\|\nabla^3\hbox{div}
v(\tau)\|_{L^2}\mathrm{d}\tau\leq CK_0.
\end{equation}
This is very remarkable, since the forth order spatial derivatives
of $u^\pm$ have no such good estimate. We think that this phenomenon
should owe to the special structure of the system. It should be
mentioned that this key observation plays an essential role in our
analysis. Next, we rewrite the term $\mathcal{R}_{3,2}$ as follows:
\begin{equation}\label{1.32}
\mathcal{R}_{3,2}=\displaystyle\int_{\mathbb{R}^3}\mathcal
C\rho^+\nabla^3 n^{-}\nabla^3
\left(\frac{(2\mu^\mp+\lambda^\mp)\text{div}u^\mp\pm\text{div}v}{2\mu^\pm+\lambda^\pm}\right)\mathrm{d}x.\end{equation}
Finally, by virtue of \eqref{1.31} and \eqref{1.32}, we can control
the trouble term $\mathcal{R}_{3,2}$ appropriately. We refer to the
proof of \eqref{3.45} for details.

\subsection{Notations and conventions}
Throughout this paper, we use $H^k(\mathbb R^3)$ to denote the usual
Sobolev space with norm $\|\cdot\|_{H^k}$ and $L^p$, $1\leq p\leq
\infty$ to denote the usual $L^p(\mathbb R^3)$ space with norm
$\|\cdot\|_{L^p}$.  For the sake of conciseness, we do not precise
in functional space names when they are concerned with
scalar--valued or vector--valued functions, $\|(f, g)\|_X$ denotes
$\|f\|_X+\|g\|_X$.  We will employ the notation $a\lesssim b$ to
mean that $a\leq Cb$ for a universal constant $C>0$ that only
depends on the parameters coming from the problem. We denote
$\nabla=\partial_x=(\partial_1,\partial_2,\partial_3)$, where
$\partial_i=\partial_{x_i}$, $\nabla_i=\partial_i$ and put
$\partial_x^\ell f=\nabla^\ell f=\nabla(\nabla^{\ell-1}f)$.  Let
$\Lambda^s$ be the pseudo differential operator defined by
\begin{equation}\Lambda^sf=\mathfrak{F}^{-1}(|{\bf \xi}|^s\widehat f),~\hbox{for}~s\in \mathbb{R},\nonumber\end{equation}
where $\widehat f$ and $\mathfrak{F}(f)$ are the Fourier transform
of $f$. The homogenous Sobolev space $\dot{H}^s(\mathbb{R}^3)$ with
norm given by $\|f\|_{\dot{H}^s}\overset{\triangle}=\|\Lambda^s
f\|_{L^2}$. For a radial function $\phi\in C_0^\infty(\mathbb
R^3_{{\bf \xi}})$ such that $\phi({\bf \xi})=1$ when $|{\bf
\xi}|\leq \frac{\eta}{2}$ and $\phi({\bf \xi})=0$ when $|{\bf
\xi}|\geq \eta$, where $\eta$ is defined in Lemma \ref{2.1}, we
define the low--frequency part of $f$ by
$$f^l=\mathfrak{F}^{-1}[\phi({\bf \xi})\widehat f]$$
and the high--frequency part of $f$ by
$$f^h=\mathfrak{F}^{-1}[(1-\phi({\bf \xi}))\widehat f].$$
It is direct to check that $f=f^l+f^h$ if the Fourier transform of
$f$ exists.

\bigskip

\section{\leftline {\bf{Spectral analysis and linear $L^2$ estimates.}}}
\setcounter{equation}{0}
\subsection{Reformulation} In this subsection, we first reformulate the system.
Setting   \[ n^{\pm}=R^{\pm}-1\quad\text{and} \quad \theta=P-\bar P,
\]
the Cauchy problem \eqref{1.15} and \eqref{1.16} can be reformulated
as
\begin{equation}\label{2.1}
\left\{\begin{array}{l}
\partial_{t}n^++\operatorname{div} u^+=-\operatorname{div} (n^+u^+),\\
\partial_{t}\theta+\beta_1\operatorname{div}u^++\beta_2\operatorname{div}u^-=F_1,\\
\partial_{t}u^{+}+\beta_3\nabla
\theta-\nu^+_1\Delta u^+-\nu^+_2\nabla\operatorname{div} u^+=F_2, \\
\partial_{t}u^{-}+\beta_4\nabla \theta-\nu^-_1\Delta u^--\nu^-_2\nabla\operatorname{div} u^-=F_3, \\
\end{array}\right.
\end{equation}
with initial data
\begin{equation}\label{2.2}
\left(n^{+}, \theta,u^{+}, u^{-}\right)(x, 0)=\left(n_{0}^{+}, \theta_0, u_{0}^{+}, u_{0}^{-}\right)(x) \rightarrow(0,0, \overrightarrow{0}, \overrightarrow{0}), \quad \text { as }|x| \rightarrow+\infty.
\end{equation}
Here $\beta_{1}=\mathcal C \left(1,\bar P\right)\rho^-(\bar P)$,
$\beta_{2}=\mathcal C \left(1,\bar P\right)\rho^+(\bar P),$
$\beta_{3}=\frac{1}{\rho^+(\bar P)}$, $\beta_{4}=\frac{1}{\rho^-(\bar P)}$,
$\nu_{1}^{\pm}=\frac{\mu^{\pm}}{\rho^{\pm}(\bar P)}$,
$\nu_{2}^{\pm}=\frac{\mu^{\pm}+\lambda^{\pm}}{\rho^{\pm}(\bar P)}>0$ and
the nonlinear terms are given by
\begin{align}
\label{2.3}F_{1}=&-g_{1}^+\left(n^{+}, \theta\right)\operatorname{div} u^+-g_{1}^-\left(n^{+}, \theta\right)\operatorname{div} u^-
-\mathcal C\rho^-u^+\cdot\nabla n^+-\mathcal C\rho^+u^-\cdot\nabla n^-, \\
F_{2}=&-u^+\cdot\nabla u^+-g_{2}^+\nabla \theta+g^{+}_3\Delta u^++g^{+}_4\nabla\operatorname{div} u^++\frac{\mu^+\left(\nabla u^{+}+\nabla^{t}
 u^{+}\right)\nabla\alpha^+}{R^+}+\frac{\lambda^+\operatorname{div} u^+\nabla\alpha^+}{R^+},\label{2.4}\\
F_{3}=&-u^-\cdot\nabla u^--g_{2}^-\nabla \theta+g^{-}_3\Delta u^-+g^{-}_4\nabla\operatorname{div} u^-+\frac{\mu^-\left(\nabla u^{-}+
\nabla^{t} u^{-}\right)\nabla\alpha^-}{R^-}+\frac{\lambda^-\operatorname{div} u^-\nabla\alpha^-}{R^-}\label{2.5},
\end{align}
where the nonlinear functions $g^{\pm}_i$$(1\leq i\leq 4)$ are defined by
\begin{equation}\label{2.6}
\left\{\begin{array}{l}
g^{+}_1\left(n^{+}, \theta\right)=\mathcal {C}(n^++1, \theta+\bar{P})\rho^-R^+-\beta_1, \\
g^{-}_1\left(n^{+}, \theta\right)=\mathcal {C}(n^++1, \theta+\bar{P})\rho^+R^--\beta_2,\\
\end{array}\right.\end{equation}
\begin{equation}\label{2.7}
\left\{\begin{array}{l}
g^{+}_2\left(\theta\right)=\frac{1}{\rho^+}-\beta_3, \\
g^{-}_2\left(\theta\right)=\frac{1}{\rho^-}-\beta_4,\\
\end{array}\right.\end{equation}
\begin{equation}\label{2.8}
g^{\pm}_3\left(\theta\right)=\frac{\mu^{\pm}}{\rho^{\pm}}-\nu_{1}^{\pm},
\end{equation}
\begin{equation}\label{2.9}
g^{\pm}_4\left(\theta\right)=\frac{\mu^{\pm}+\lambda^{\pm}}{\rho^{\pm}}-\nu_2^{\pm}.\end{equation}
The linearized system corresponding to the Cauchy problem \eqref{2.1}--\eqref{2.2} reads
\begin{equation}\label{2.11}
\left\{\begin{array}{l}
\partial_{t} n^++\operatorname{div}u^+=0, \\
\partial_{t}\theta+\beta_1\operatorname{div}u^++\beta_2\operatorname{div}u^-=0,\\
\partial_{t}u^{+}+\beta_3\nabla
\theta-\nu^+_1\Delta u^+-\nu^+_2\nabla\operatorname{div} u^+=0, \\
\partial_{t}u^{-}+\beta_4\nabla \theta-\nu^-_1\Delta u^--\nu^-_2\nabla\operatorname{div} u^-=0, \\
\left(n^{+}, \theta, u^{+},  u^{-}\right)(x, 0)=\left(n_{0}^{+}, \theta_0, u_{0}^{+},  u_{0}^{-}\right)(x).
\end{array}\right.
\end{equation}
It is easy to check that there is a zero eigenvalue for the
linearized problem \eqref{2.11}, which makes the problem become much
more difficult and complicated. Noticing that the equations
\eqref{2.11}$_2$, \eqref{2.11}$_3$ and \eqref{2.11}$_4$ are
decoupled with $n^+$, thus we consider the following Cauchy problem
for $(\theta, u^+,u^-)$, which enables us to exclude the case of
zero eigenvalue
\begin{equation}\label{2.12}
\left\{\begin{array}{l}
\partial_{t}\theta+\beta_1\operatorname{div}u^++\beta_2\operatorname{div}u^-=0,\\
\partial_{t}u^{+}+\beta_3\nabla
\theta-\nu^+_1\Delta u^+-\nu^+_2\nabla\operatorname{div} u^+=0, \\
\partial_{t}u^{-}+\beta_4\nabla \theta-\nu^-_1\Delta u^--\nu^-_2\nabla\operatorname{div} u^-=0, \\
\left(\theta,  u^{+}, u^{-}\right)(x, 0)=\left(\theta_0, u_{0}^{+},  u_{0}^{-}\right)(x).
\end{array}\right.
\end{equation}

 In terms of the semigroup theory
for evolutionary equation, the solution $\left(\theta,u^+,u^-\right)$ of linear Cauchy problem \eqref{2.12} can
be expressed via the Cauchy problem for $U=\left(\theta,u^+,u^-\right)^t$ as
\begin{equation}
\begin{cases}
{U}_t=\mathcal B{U},\\
{U}\big|_{t=0}={U}_0,
\end{cases}   \label{2.14}
\end{equation}
where the operator $\mathcal B$ is given by
\begin{equation}\nonumber\mathcal B=\begin{pmatrix}
0&-\beta_1\text{div}&-\beta_2\text{div}\\
-\beta_3\nabla&\nu_{1}^{+} \Delta+\nu_{2}^{+} \nabla \text{div}&0\\
-\beta_4\nabla&0&\nu_{1}^{-} \Delta+\nu_{2}^{-} \nabla \text{div}
\end{pmatrix}.\end{equation}
Applying Fourier transform to the system \eqref{2.14}, one has
\begin{equation}
\begin{cases}
\widehat {{U}}_t=\mathcal A(\xi)\widehat {U},\\
\widehat {U}\big|_{t=0}=\widehat U_0,
\end{cases}   \label{2.15}
\end{equation}
where $\widehat {U}(\xi,t)=\mathfrak{F}(U(x,t))$, $\xi=(\xi^1,\xi^2,\xi^3)^t$ and $\mathcal A(\xi)$ is defined by
\begin{equation}\nonumber\mathcal A(\xi)=\begin{pmatrix}
0&-{i}\beta_1\xi^t&-{i}\beta_2\xi^t\\
-{i}\beta_3\xi&-\nu_1^+|\xi|^2{\rm I}_{3\times 3}-\nu_2^+\xi\otimes\xi&0\\
- {i}\beta_4\xi&0&-\nu_1^-|\xi|^2{\rm I}_{3\times 3}-\nu_2^-\xi\otimes\xi
\end{pmatrix}.\end{equation}

To derive the linear time--decay estimates,  by using a real method
as in \cite{Kobayashi2, Mat1, Mat2}, one need to make a detailed
analysis on the properties of the semigroup. To simplify the
analysis of the Green function which is a $7\times 7$ system, we
will employ the Hodge decomposition technique firstly introduced by
Danchin \cite{Dan1} to split the linear system into three systems.
One only has three equations and  its characteristic polynomial
possesses three distinct roots, the other two systems are the heat
equation. This key observation allows us to get the optimal linear
convergence rates.

 Let
$\varphi^{\pm}=\Lambda^{-1}{\rm div}{u}^{\pm}$
be the ``compressible part" of the velocities ${u}^{\pm}$, and denote $\phi^{\pm}=\Lambda^{-1}{\rm curl}{u}^{\pm}$ (with $({\rm curl} z)_i^j
=\partial_{x_j}z^i-\partial_{x_i}z^j$) by the ``incompressible part"
of the velocities ${u}^{\pm}$. Then, we can rewrite
the system \eqref{2.12} as follows:
\begin{equation}\label{2.11-1}
\begin{cases}
\partial_t{\theta}+\beta_1\Lambda{\varphi^+}+\beta_2\Lambda{\varphi^-}=0,\\
\partial_t{\varphi^+}-\beta_3\Lambda\theta+\nu^+\Lambda^2{\varphi^+}=0,\\
\partial_t{\varphi^-}-\beta_4\Lambda\theta+\nu^-\Lambda^2{\varphi^-}=0,\\
(\theta, \varphi^+,  \varphi^-)\big|_{t=0}=(\theta_0, \Lambda^{-1}{\rm div}{u}^{+}_0,  \Lambda^{-1}{\rm div}{u}^{-}_0)(x),\\
\end{cases}
\end{equation}
and
\begin{equation}\label{2.12-1}
\begin{cases}
\partial_t\phi^\pm+\nu^\pm_1\Lambda^2\phi^\pm=0,\\
\phi^\pm\big|_{t=0}=\Lambda^{-1}{\rm curl}{u}^{\pm}_0(x),
\end{cases}
\end{equation}
where $\nu^{\pm}=\nu^{\pm}_1+\nu^{\pm}_2$.
\smallskip
\subsection{Spectral analysis for IVP \eqref{2.11-1}} In terms of the semigroup theory, we may represent the
 IVP \eqref{2.11-1} for $\mathcal U=(\theta, \varphi^+,  \varphi^-)^t$ as
\begin{equation}
\begin{cases}
\mathcal U_t=\mathcal B_1\mathcal U,\\
\mathcal U\big|_{t=0}=\mathcal U_0,
\end{cases}   \label{2.13}
\end{equation}
where the operator $\mathcal B_1$ is defined
by
\begin{equation}\nonumber\mathcal B_1=\begin{pmatrix}
0&-\beta_1\Lambda&-\beta_2\Lambda\\
\beta_3\Lambda&-\nu^+\Lambda^2&0\\
\beta_4\Lambda&0&-\nu^-\Lambda^2
\end{pmatrix}.\end{equation}
Taking the Fourier transform to the system \eqref{2.13}, we obtain
\begin{equation}
\begin{cases}
\widehat {\mathcal U}_t=\mathcal A_1(\xi)\widehat {\mathcal U},\\
\widehat {\mathcal U}\big|_{t=0}=\widehat {\mathcal U}_0,
\end{cases}   \label{2.14-1}
\end{equation}
where $\widehat {\mathcal U}(\xi,t)=\mathfrak{F}({\mathcal U}(x,t))$  and $\mathcal A_1(\xi)$ is given by
\begin{equation}\nonumber\mathcal A_1(\xi)=\begin{pmatrix}
0&-\beta_1|\xi|&-\beta_2|\xi|\\
\beta_3|\xi|&-\nu^+|\xi|^2&0\\
\beta_4|\xi|&0&-\nu^-|\xi|^2
\end{pmatrix}.\end{equation}
We compute the eigenvalues of matrix $\mathcal A_1(\xi)$  from the
determinant
\begin{equation}\begin{split}\label{2.15-1}&{\rm det}(\lambda{\rm I}-\mathcal A_1(\xi))\\
&=\lambda^3+(\nu^++\nu^-)|\xi|^2\lambda^2+(\beta_1\beta_3+\beta_2\beta_4+\nu^+\nu^-|\xi|^2)
|\xi|^2\lambda+(\beta_1\beta_3\nu^-+\beta_2\beta_4\nu^+)|\xi|^4,
\end{split}\end{equation}
which implies that matrix $\mathcal A_1(\xi)$ possesses three
different eigenvalues:
\begin{equation}\nonumber
 \lambda_1=\lambda_1(|\xi|),\quad \lambda_2=\lambda_2(|\xi|),\quad
 \lambda_3=\lambda_3(|\xi|).
\end{equation}
Consequently, the semigroup $\text{e}^{t\mathcal A_1}$ can be decomposed
into
\begin{equation}\label{2.16}
 \text{e}^{t\mathcal A_1(\xi)}=\sum_{k=1}^3\text{e}^{\lambda_kt}P_k(\xi),
\end{equation}
where the projector $P_k(\xi)$ is defined by
\begin{equation}\label{2.17}
 P_k(\xi)=\prod_{j\neq k}\frac{\mathcal A_1(\xi)-\lambda_jI}{\lambda_k-\lambda_j}, \quad k,j=1,2,3.
\end{equation}
Thus, the solution of IVP \eqref{2.14-1} can be expressed as
\begin{equation}
\widehat {\mathcal U}(\xi,t)=\text{e}^{t\mathcal A_1(\xi)}\widehat {\mathcal U}_0(\xi)=\left(\sum_{k=1}^3
\text{e}^{\lambda_kt}P_k(\xi)\right)\widehat {\mathcal U}_0(\xi).\label{2.18}
\end{equation}

To derive long time properties of the semigroup $\text{e}^{t\mathcal
A_1}$ in $L^2$ framework, one need to analyze the asymptotical
expansions of $\lambda_k$, $P_k$ $(k =1, 2, 3)$ and
$\text{e}^{t\mathcal A_1(\xi)}$ in the low frequency part. Employing
the similar argument of Taylor series expansion as in
\cite{Kobayashi2, Mat1, Mat2}, we have from tedious calculations
that
\begin{Lemma}\label{lemma2.1}
There exists a positive constants $\eta\ll 1 $ such that, for
$|\xi|\leq \eta$, the spectral has the following Taylor series
expansion
\begin{equation}\label{2.19}
\left\{\begin{array}{lll}\displaystyle \lambda_1=-\frac{\beta_1\beta_3\nu^++\beta_2\beta_4\nu^-}{2(\beta_1\beta_3+\beta_2\beta_4)}|\xi|^2+\sqrt{\beta_1\beta_3+\beta_2\beta_4}\text{i}|\xi|+\mathcal O(|\xi|^3),\\
\displaystyle\lambda_2=-\frac{\beta_1\beta_3\nu^++\beta_2\beta_4\nu^-}{2(\beta_1\beta_3+\beta_2\beta_4)}|\xi|^2-\sqrt{\beta_1\beta_3+\beta_2\beta_4}\text{i}|\xi|+\mathcal O(|\xi|^3),
\\ \displaystyle  \lambda_3=-\frac{\beta_1\beta_3\nu^-+\beta_2\beta_4\nu^+}{\beta_1\beta_3+\beta_2\beta_4}|\xi|^2+\mathcal O(|\xi|^3).
\end{array}\right.
\end{equation}
\end{Lemma}

By virtue of \eqref{2.17}--\eqref{2.19}, we can establish the
following estimates for the low--frequent part of the solutions
$\widehat{\mathcal U}(t,\xi)$ to the IVP \eqref{2.14-1}:
\begin{Lemma}\label{lemma2.2} Let $\displaystyle\bar{\nu}=\min\left\{\frac{\beta_1\beta_3\nu^++\beta_2\beta_4\nu^-}{2(\beta_1\beta_3+\beta_2\beta_4)},\frac{\beta_1\beta_3\nu^-+\beta_2\beta_4\nu^+}{\beta_1\beta_3+\beta_2\beta_4}\right\}>0$,
 we have
\begin{equation}\label{2.20}
 |\widehat{\theta}|,~|\widehat{\varphi^+}|,~ |\widehat{\varphi^-}|\lesssim \text{e}^{-\bar{\nu}|\xi|^2t}
 ( |\widehat{\theta_0}|+|\widehat{\varphi^+_0}|+|\widehat{\varphi^-_0}|),
\end{equation}
for any $|\xi|\le \eta$.
\end{Lemma}
\begin{proof} By virtue of formula \eqref{2.17} and Taylor series expansion
of $\lambda_k$ $(k=1,2,3)$ in \eqref{2.19}, we can represent $P_k$ $( k=1,2,3)$ as follows:
\begin{equation}\label{2.21}\begin{split}P_1(\xi)=&\begin{pmatrix}
\frac{1}{2}&
\frac{\beta_1}{2\sqrt{\beta_1\beta_3+\beta_2\beta_4}}\text{i}&\frac{\beta_2}{2\sqrt{\beta_1\beta_3+\beta_2\beta_4}}\text{i}\\
-\frac{\beta_3}{2\sqrt{\beta_1\beta_3+\beta_2\beta_4}}\text{i}&\frac{\beta_1\beta_3}{2\left(\beta_1\beta_3+\beta_2\beta_4\right)}
&\frac{\beta_2\beta_3}{2\left(\beta_1\beta_3+\beta_2\beta_4\right)}\\
-\frac{\beta_4}{2\sqrt{\beta_1\beta_3+\beta_2\beta_4}}\text{i}&\frac{\beta_1\beta_4}{2\left(\beta_1\beta_3+\beta_2\beta_4\right)}
&\frac{\beta_2\beta_4}{2\left(\beta_1\beta_3+\beta_2\beta_4\right)}
\end{pmatrix}+\mathcal O(|\xi|),\end{split}\end{equation}

\begin{equation}\label{2.22}\begin{split}P_2(\xi)=&\begin{pmatrix}
\frac{1}{2}&
-\frac{\beta_1}{2\sqrt{\beta_1\beta_3+\beta_2\beta_4}}\text{i}&-\frac{\beta_2}{2\sqrt{\beta_1\beta_3+\beta_2\beta_4}}\text{i}\\
\frac{\beta_3}{2\sqrt{\beta_1\beta_3+\beta_2\beta_4}}\text{i}&\frac{\beta_1\beta_3}{2\left(\beta_1\beta_3+\beta_2\beta_4\right)}
&\frac{\beta_2\beta_3}{2\left(\beta_1\beta_3+\beta_2\beta_4\right)}\\
\frac{\beta_4}{2\sqrt{\beta_1\beta_3+\beta_2\beta_4}}\text{i}&\frac{\beta_1\beta_4}{2\left(\beta_1\beta_3+\beta_2\beta_4\right)}
&\frac{\beta_2\beta_4}{2\left(\beta_1\beta_3+\beta_2\beta_4\right)}
\end{pmatrix}+\mathcal O(|\xi|),\end{split}\end{equation}

and
\begin{equation}\label{2.23}\begin{split}P_3(\xi)=&\begin{pmatrix}
0&
0&0\\
0&\frac{\beta_2\beta_4}{\beta_1\beta_3+\beta_2\beta_4}
&-\frac{\beta_2\beta_3}{\beta_1\beta_3+\beta_2\beta_4}\\
0&-\frac{\beta_1\beta_4}{\beta_1\beta_3+\beta_2\beta_4}
&\frac{\beta_1\beta_3}{\beta_1\beta_3+\beta_2\beta_4}
\end{pmatrix}+\mathcal O(|\xi|),\end{split}\end{equation}
for any $|\xi|\le \eta$. Therefore, \eqref{2.20} follows from
\eqref{2.18}--\eqref{2.19} and \eqref{2.21}--\eqref{2.23} immediately.
\end{proof}

With the help of \eqref{2.20}, we can get the following Proposition which is concerned with
the optimal $L^2$ convergence rate on the low--frequency part of the
solution.
\begin{Proposition}[$L^2$--theory]\label{Prop2.3} Let $k\ge 0$ and  $1\le p \le 2$, it holds that
\begin{equation}\|\nabla ^k\text{e}^{t\mathcal B_1}\mathcal U^l(0)\|_{L^2}\lesssim(1+t)^{-\frac{3}{2}\left(\frac{1}{p}-\frac{1}{2}\right)-\frac{k}{2}}\| {\mathcal U}(0)\|_{L^{p}},\label{2.24}\end{equation}
 for any $t\geq 0$.
\end{Proposition}
\begin{proof} Due to \eqref{2.20} and Plancherel theorem, we have
\begin{equation}\nonumber \begin{split}\|\nabla ^k\text{e}^{t\mathcal B_1}\mathcal U^l(0)\|_{L^2}^2=
&~\big{\|}|\xi|^k\text{e}^{t\mathcal{A}_1(\xi)}\widehat{\mathcal U}^l(0)\big{\|}_{L^2}^2
\\ \lesssim &\int_{|\xi|\le \eta} \text{e}^{-2\bar{\nu}|\xi|^2t}|\xi|^{2k}|\widehat{\mathcal U}^l(0)|^2\mathrm{d}\xi \\
\lesssim &~(1+t)^{-3\left(\frac{1}{p}-\frac{1}{2}\right)-k}\|\widehat {\mathcal U}^l(0)\|_{L^{q}}^2
\\
\lesssim &~(1+t)^{-3\left(\frac{1}{p}-\frac{1}{2}\right)-k}\| {\mathcal U}(0)\|_{L^{p}}^2,
\end{split}\end{equation}
where $q$ satisfies $\frac{1}{p}+\frac{1}{q}=1$. This implies
\eqref{2.24} immediately, and thus the proof of Proposition
\ref{Prop2.3}is completed.
\end{proof}

It should be mentioned that the $L^2$--convergence rates derived
above are optimal. Indeed, we have the lower--bound on the
convergence rates which is stated in the following Proposition \ref{Prop2.3}.

\begin{Proposition}\label{Prop2.4} Assume that $\left(\widehat{\theta},\widehat{\varphi^+}, \widehat{\varphi^-}\right)\in L^1$ satisfies \begin{equation}\widehat{\varphi^+_0}(\xi)=\widehat{\varphi^-_0}(\xi)=0 \quad \text{and} \quad |\widehat{{\theta}_0}(\xi)|\ge c_0,\label{2.25}\end{equation} for any $|\xi|\le\eta$. Then it holds that the global
solution $(\theta, \varphi^+,  \varphi^-)$ of the IVP \eqref{2.13}
satisfies
\begin{equation}\min\left\{\|\theta^l\|_{L^2},\|{\varphi}^{+,l}\|_{L^2},\|{\varphi}^{-,l}\|_{L^2}\right\}\gtrsim
c_0(1+t)^{-\frac{3}{4}}.\label{2.26}\end{equation}
for large enough $t$ .
\end{Proposition}
\begin{proof} Let $\bar{\nu}_1=\frac{\beta_1\beta_3\nu^++\beta_2\beta_4\nu^-}{2(\beta_1\beta_3+\beta_2\beta_4)}>0$. Due to \eqref{2.25},
it follows from \eqref{2.18}--\eqref{2.19} and
\eqref{2.21}--\eqref{2.23} that
\begin{equation}\begin{split}\nonumber\widehat{\theta}^{l}\sim \text{e}^{-\bar{\nu}_1|\xi|^2t}\cos\left(\sqrt{\beta_1\beta_3+\beta_2\beta_4}|\xi|t+\mathcal O(|\xi|^3)t\right)\widehat {\theta}_0^{l}\end{split}\end{equation}
This together with Plancherel theorem and the double angle formula
gives that
\begin{equation}\label{2.27}\begin{split}\|{\theta}^{l}\|_{L^2}^2=&~\|\widehat{\theta^l}\|_{L^2}^2\\
\ge
&~\frac{c_0^2}{2}\int_{|\xi|\le \eta}\text{e}^{-2\bar{\nu}_1|\xi|^2t}\cos^2
\left(\sqrt{\beta_1\beta_3+\beta_2\beta_4}|\xi|t+\mathcal O(|\xi|^3)t\right)\mathrm{d}\xi\\
=&~\frac{c_0^2}{4}\int_{|\xi|\le
\eta}\text{e}^{-2\bar{\nu}_1|\xi|^2t}{d}\xi+\frac{c_0^2}{4}\int_{|\xi|\le
\eta}\text{e}^{-\bar{\nu}_1|\xi|^2t}\cos
\left(2\sqrt{\beta_1\beta_3+\beta_2\beta_4}|\xi|t+\mathcal
O(|\xi|^3)t\right)\mathrm{d}\xi
\\ \ge &~\frac{c_0^2}{4}(1+t)^{-\frac{3}{2}}-Cc_0^2(1+t)^{-\frac{3}{2}}(1+t)^{-\frac{1}{2}}
 \\~ \gtrsim &~c_0^2(1+t)^{-\frac{3}{2}},\end{split}\end{equation}
 if $t$ is large enough.
 Using a similar procedure as in \eqref{2.27} to handle $\|(\varphi^{+,l},\varphi^{-,l})\|_{L^2}$, one has \eqref{2.26}. Therefore, we have completed the proof of
 Proposition \ref{Prop2.4}.
\end{proof}

\smallskip
\subsection{Spectral analysis for IVP \eqref{2.12-1}}

From the classic theory of the heat equation, it is clear that the
solution $\phi^\pm$  to the IVP \eqref{2.12-1}  satisfies the
following decay estimate.
\begin{Proposition}[$L^2$--theory]\label{Prop2.5} Let $k\ge 0$ and   $1\le p \le 2$, it holds that
\begin{equation}\|\nabla ^k\text{e}^{-\nu^\pm_1\Lambda^2 t}\mathcal \phi^{\pm,l}(0)\|_{L^2}\lesssim(1+t)^{-\frac{3}{2}\left(\frac{1}{p}-\frac{1}{2}\right)-\frac{k}{2}}\| {\phi^{\pm}}(0)\|_{L^{p}},\label{2.28}\end{equation}
 for any $t\geq 0$.
\end{Proposition}

\smallskip
\subsection{$L^2$ decay estimates for IVP \eqref{2.14}}

By virtue of the definition of $\varphi^{\pm}$ and $\phi^{\pm}$, and the
fact that the relations
$${u}^{\pm}=-\wedge^{-1}\nabla\varphi^{\pm}-\wedge^{-1}\text{div}\phi^{\pm}$$
involves pseudo--differential operators of degree zero, the
estimates in space $H^k(\mathbb R^3)$ for the original function
${u}^{\pm}$ will be the same as for $(\varphi^{\pm}, \phi^{\pm})$.
Combining Propositions \ref{Prop2.3}, \ref{Prop2.4} and
\ref{Prop2.5}, we have the following result concerning long time
properties for the solution semigroup $\text{e}^{-t\mathcal{A}}$.
\begin{Proposition}\label{Prop2.6} Let $k\ge 0$ and  $1\le p \le 2$. Assume
that the initial data $U_0\in L^p(\mathbb R^3)$,   then for any
$t\ge 0$, the global solution ${U}=(\theta, {u}^+, {u}^-)^t$ of the
IVP \eqref{2.12} satisfies
\begin{equation}\|\nabla ^k\text{e}^{t\mathcal{B}}{U}^l(0)\|_{L^2}\leq
C(1+t)^{-\frac{3}{2}\left(\frac{1}{p}-\frac{1}{2}\right)-\frac{k}{2}}\| {U}(0)\|_{L^{p}}.\label{2.29}\end{equation}
If additionally the initial
data satisfies \eqref{2.25}, we also have the following lower--bound
on convergence rate
\begin{equation}\min\left\{\|\theta^l(t)\|_{L^2},\|{u}^{+,l}(t)\|_{L^2},\|{u}^{-,l}(t)\|_{L^2}\right\}\geq C_1c_0(1+t)^{-\frac{3}{4}}\label{2.30}
\end{equation}
 if $t$ large enough.
\end{Proposition}

\bigskip

\section{\leftline {\bf{Proof of Theorem \ref{1mainth}.}}}
\setcounter{equation}{0}

By a classic argument, the global existence of solutions will be
obtained by combining the local existence result with a priori
estimates. Since the local classical solutions can be proved by a
standard argument of the Lax--Milgram theorem and
Schauder--Tychonoff fixed--point theorem as in \cite{Mat1,Mat2}
whose details are omitted, global existence of classical solutions
will follow in a standard continuity argument after we have
established a priori estimates. Therefore, we assume a priori that
\begin{equation}\label{3.1}\|\left(n^+,\theta,u^+,n^-,u^-\right)\|_{H^3}\le \delta\ll 1,
\end{equation}
here $\delta\sim\delta_0$ is small enough. This together with
Sobolev inequality in \cite{Evans, Ziemer} particularly implies that
\begin{equation}\label{3.2}\|\left(n^+,\theta,u^+,n^-,u^-\right)\|_{W^{1,\infty}}\lesssim \delta.
\end{equation}

In what follows, a series of lemmas on the energy estimates are
given. First, the zero--order energy estimate of
$\left(\theta,u^+,u^-\right)$ is obtained in the following lemma.
\smallskip
\begin{Lemma}\label{Lemma3.1}Assume that the notations and hypotheses  of Theorem \ref{1mainth} and \eqref{3.1}  are in force, then
\begin{equation}\label{3.3}\begin{split}\frac{1}{2}\frac{\rm{d}}{{\rm{d}}t}&\left\{\left|\left|\theta\right|\right|^2_{L^2}+\frac{\beta_1}{\beta_3}\left|\left|u^+\right|\right|^2_{L^2}
+\frac{\beta_2}{\beta_4}\left|\left|u^-\right|\right|^2_{L^2}\right\}
+\frac{3\beta_1}{4\beta_3}\left(\nu_1^+\left|\left|\nabla u^+\right|\right|^2_{L^2}+\nu_2^+\left|\left|\text{\rm div}u^+\right|\right|^2_{L^2}\right)\\
&\quad\hspace{0.86cm}+\frac{3\beta_2}{4\beta_4}\left(\nu_1^-\left|\left|\nabla u^-\right|\right|^2_{L^2}+\nu_2^-\left|\left|\text{\rm div}u^-\right|\right|^2_{L^2}\right)\le C\delta\left|\left|\nabla \theta\right|\right|^2_{L^2}.
\end{split}\end{equation}
\end{Lemma}
\begin{proof}Multiplying \eqref{2.1}$_2$, \eqref{2.1}$_3$, \eqref{2.1}$_4$ by $\theta$, $\frac{\beta_1}{\beta_3}u^+$, $\frac{\beta_2}{\beta_4}u^-$, respectively, summing up and then integrating
the resultant equation over $\mathbb R^3$, we have
\begin{align}\begin{split}\label{3.4}
&\frac{1}{2}\frac{\rm{d}}{{\rm{d}}t}\left\{\left|\left|\theta\right|\right|^2_{L^2}+\frac{\beta_1}{\beta_3}\left|\left|u^+\right|\right|^2_{L^2}
+\frac{\beta_2}{\beta_4}\left|\left|u^-\right|\right|^2_{L^2}\right\}\\
&\quad+\frac{\beta_1}{\beta_3}\left(\nu_1^+\left|\left|\nabla u^+\right|\right|^2_{L^2}+\nu_2^+\left|\left|\text{\rm div}u^+\right|\right|^2_{L^2}\right)+\frac{\beta_2}{\beta_4}\left(\nu_1^-\left|\left|\nabla u^-\right|\right|^2_{L^2}+\nu_2^-\left|\left|\text{\rm div}u^-\right|\right|^2_{L^2}\right)\\
&=\displaystyle\int_{\mathbb{R}^3}\left(
{{F}}_{1}\theta+\frac{\beta_1}{\beta_3} {F}_{2}u^{+}+
\frac{\beta_2}{\beta_4}F_3u^{-}\right)\mathrm{d}x \\
&=:I_{1}+I_{2}+I_{3}.
\end{split}\end{align}
Using  \eqref{3.1}, integration by parts, Lemma
\ref{1interpolation}, H\"{o}lder inequality, Sobolev inequality and
Young inequality, we have
\begin{align}\begin{split}\label{3.5}
|I_1|&\lesssim \displaystyle \int_{\mathbb{R}^3}|(n^+,
\theta)|\cdot|\nabla(u^+,
u^-)|\cdot|\theta|\mathrm{d}x\\
&\quad\displaystyle+\int_{\mathbb{R}^3}|\mathcal{C}\rho^\mp||n^{\pm}||u^\pm||\nabla
\theta|+|\mathcal{C}\rho^\mp||n^{\pm}||\hbox{div}u^\pm||\theta|
+|\nabla(\mathcal{C}\rho^\mp)||n^{\pm}||u^\pm||\theta|\mathrm{d}x\\
&\lesssim \left|\left|(n^+,
\theta)\right|\right|_{L^3}\left|\left|\nabla(u^+,
u^-)\right|\right|_{L^2}\left|\left|\theta\right|\right|_{L^6}+\left|\left|\nabla
\theta\right|\right|_{L^2}\left|\left|(u^+,
u^-)\right|\right|_{L^6}\left|\left|(n^+,
n^-)\right|\right|_{L^3}\\&\quad+\displaystyle \left|\left|
\theta\right|\right|_{L^6}\left|\left|(\hbox{div}u^+,
\hbox{div}u^-)\right|\right|_{L^2}\left|\left|(n^+,
n^-)\right|\right|_{L^3}
\\&\quad+\displaystyle \left|\left|
\theta\right|\right|_{L^6}\left|\left|(u^+,
u^-)\right|\right|_{L^6}\left|\left|(n^+,
n^-)\right|\right|_{L^3}\left|\left|\nabla(n^+,
\theta)\right|\right|_{L^3}
\\
&\lesssim \delta\left|\left|\nabla(\theta, u^+, u^-)\right|\right|_{L^2}^2.
\end{split}\end{align}

\noindent For the term $I_2$, by employing integration by parts,
H\"{o}lder inequality, Sobolev inequality and Young inequality, we
obtain
\begin{align}\begin{split}\label{3.6}
|I_2|&\lesssim \displaystyle \int_{\mathbb{R}^3}\left(|u^+||\nabla u^+|+|\theta||\nabla \theta|+
|\nabla n^+||\nabla u^+|+|\nabla \theta||\nabla u^+|\right)|u^+|+|\theta||\nabla u^+|^2{\mathrm{d}x}\\
&\lesssim \left|\left|u^+\right|\right|_{L^6}\left|\left|\nabla
u^+\right|\right|_{L^2}\left|\left|u^+\right|\right|_{L^3}+
\left|\left|\theta\right|\right|_{L^6}\left|\left|\nabla \theta\right|\right|_{L^2}\left|\left|u^+\right|\right|_{L^3}+\left|\left|\nabla n^+\right|\right|_{L^3}\left|\left|\nabla u^+\right|\right|_{L^2}\left|\left|u^+\right|\right|_{L^6}\\
&\quad+\left|\left|\nabla\theta\right|\right|_{L^2}\left|\left|\nabla u^+\right|\right|_{L^2}\left|\left|u^+\right|\right|_{L^\infty}+\left|\left|\theta\right|\right|_{L^\infty}\left|\left|\nabla u^+\right|\right|_{L^2}^2\\
&\lesssim \delta\left|\left|\nabla(\theta, u^+)\right|\right|_{L^2}^2.
\end{split}\end{align}
Similarly, we have
\begin{equation}\label{3.7}|I_3|\lesssim \delta\left|\left|\nabla(\theta, u^-)\right|\right|_{L^2}^2.\end{equation}
Substituting \eqref{3.5}--\eqref{3.7} into \eqref{3.4} and using the
smallness of $\delta$ yield \eqref{3.3} immediately. Therefore, we
complete the proof of Lemma \ref{Lemma3.1}.
\end{proof}
Next, we deduce energy estimates on the first and second order
spatial derivatives of $\left(\theta,u^+,u^-\right)$, which are
stated in the following lemma.
\begin{Lemma}\label{Lemma3.2}Assume that the notations and hypotheses  of Theorem \ref{1mainth} and \eqref{3.1}  are in force, then
\begin{equation}\label{3.8}\begin{split}\frac{1}{2}\frac{\rm{d}}{{\rm{d}}t}&\left\{\left|\left|\nabla\theta\right|\right|^2_{H^1}
+\frac{\beta_1}{\beta_3}
\left|\left|\nabla u^+\right|\right|^2_{H^1}
+\frac{\beta_2}{\beta_4}\left|\left|\nabla u^-\right|\right|^2_{H^1}\right\}
+\frac{3\beta_1}{4\beta_3}\left(\nu_1^+\left|\left|\nabla^2 u^+\right|\right|^2_{H^1}+\nu_2^+\left|\left|\nabla\text{\rm div}u^+\right|\right|^2_{H^1}\right)\\
&\quad\hspace{0.86cm}+\frac{3\beta_2}{4\beta_4}\left(\nu_1^-\left|\left|\nabla^2 u^-\right|\right|^2_{H^1}+\nu_2^-\left|\left|\nabla\text{\rm div}u^-\right|\right|^2_{H^1}\right)\le C\delta\left(\left|\left|\nabla\theta\right|\right|_{H^1}^2+\left|\left|\nabla(u^+,u^-)\right|\right|_{H^2}^2\right).
\end{split}\end{equation}
\end{Lemma}
\begin{proof}For $k=1,2$, multiplying $\nabla^k$\eqref{2.1}$_2$, $\nabla^k$\eqref{2.1}$_3$, $\nabla^k$\eqref{2.1}$_4$ by $\nabla^k\theta$, $\frac{\beta_1}{\beta_3}\nabla^ku^+$, $\frac{\beta_2}{\beta_4}\nabla^ku^-$, respectively, summing up and then integrating
the resultant equation over $\mathbb R^3$, we have
\begin{align}\begin{split}\label{3.9}
&\frac{1}{2}\frac{\rm{d}}{{\rm{d}}t}\left\{\left|\left|\nabla^k\theta\right|\right|^2_{L^2}+\frac{\beta_1}{\beta_3}\left|\left|\nabla^ku^+\right|\right|^2_{L^2}
+\frac{\beta_2}{\beta_4}\left|\left|\nabla^ku^-\right|\right|^2_{L^2}\right\}\\
&\quad+\frac{\beta_1}{\beta_3}\left(\nu_1^+\left|\left|\nabla^{k+1} u^+\right|\right|^2_{L^2}+\nu_2^+\left|\left|\nabla^k\text{\rm div}u^+\right|\right|^2_{L^2}\right)+\frac{\beta_2}{\beta_4}\left(\nu_1^-\left|\left|\nabla^{k+1} u^-\right|\right|^2_{L^2}+\nu_2^-\left|\left|\nabla^k\text{\rm div}u^-\right|\right|^2_{L^2}\right)\\
&=\displaystyle\int_{\mathbb{R}^3}\left(
\nabla^k{{F}}_{1}\cdot\nabla^k\theta+\frac{\beta_1}{\beta_3}\nabla^k{F}_{2}\cdot
\nabla^k u^{+}+\frac{\beta_2}{\beta_4}
\nabla^k F_3\cdot\nabla^ku^{-}\right)\mathrm{d}x\\
&=:J_{1}^k+J_{2}^k+J_{3}^k.
\end{split}\end{align}
By using integration by parts, \eqref{3.1},  Lemmas
\ref{1interpolation}--\ref{es-product}, H\"{o}lder inequality,
Sobolev inequality and Young inequality, we have
\begin{align}\begin{split}\label{3.10}
|J_{1}^k|&\leq\displaystyle \int_{\mathbb{R}^3}
 \left(\left|\nabla^k\theta\right|\left|\nabla^k\left({g^+_1(n^+, \theta)\textrm{div}u^+}\right)\right|
+\displaystyle \left|\nabla^k\theta\right|\left|\nabla^k\left({g^-_1(n^+, \theta)\textrm{div}u^-}\right)\right|\right.\\
&\quad+\left.\left|\nabla^{k}\theta
\right|\left|\nabla^{k}\left({\mathcal{C}\rho^-u^+\cdot \nabla
n^+}\right)\right|
+\left|\nabla^{k}\theta\right|\left| \nabla^{k}\left({\mathcal{C}\rho^+u^-\cdot \nabla n^-}\right)\right|\right)\mathrm{d}x\\
&\lesssim \Big(\left|\left|(n^+,
\theta)\right|\right|_{L^\infty}\left|\left|\nabla(u^+,
u^-)\right|\right|_{H^k}+ \left|\left|\nabla(u^+,
u^-)\right|\right|_{L^\infty}\left|\left|(n^+,
\theta)\right|\right|_{H^k}
\\&\quad+\left|\left|\nabla(n^+, n^-)\right|\right|_{L^\infty}\left|\left|(u^+, u^-)\right|\right|_{L^\infty}
\left|\left|\nabla(\mathcal{C}\rho^+, \mathcal{C}\rho^-)\right|\right|_{H^{k-1}}\\
& \quad +\left|\left|(\mathcal{C}\rho^+, \mathcal{C}\rho^-)\right|\right|_{L^\infty}\big(\left|\left|(u^+, u^-)\right|\right|_{L^\infty}\left|\left|\nabla (n^+,  n^-)\right|\right|_{H^k}\\
&\quad+\left|\left|\nabla (u^+,  u^-)\right|\right|_{H^1}\left|\left|\nabla (n^+,  n^-)\right|\right|_{L^\infty}\big)\Big)\left|\left|\nabla^k\theta\right|\right|_{L^2}\\
&\lesssim \delta\left(\left|\left|\nabla\theta\right|\right|_{H^1}^2+\left|\left|\nabla(u^+,u^-)\right|\right|_{H^2}^2\right).
\end{split}\end{align}

For the terms $J_2^k$, by employing integration by parts, Lemmas
\ref{1interpolation}--\ref{es-product}, H\"{o}lder inequality,
Sobolev inequality and Young inequality, we obtain
\begin{align}\begin{split}\label{3.11}
|J_{2}^k|&= \displaystyle -\frac{\beta_1}{\beta_3}\int_{\mathbb{R}^3}\nabla^{k-1}{F}_{2}\cdot\nabla^{k+1} u^{+}\mathrm{d}x\\
&\lesssim \left|\left|u^+\right|\right|_{L^\infty}\left|\left|\nabla u^+\right|\right|_{L^2}\left|\left|\nabla^2 u^+\right|\right|_{L^2}+\left|\left|\theta\right|\right|_{L^\infty}\left|\left|\nabla \theta\right|\right|_{L^2}\left|\left|\nabla^2 u^+\right|\right|_{L^2}\\
&\quad+ \left|\left|\theta\right|\right|_{L^\infty}\left|\left|\Delta u^+\right|\right|_{L^2}\left|\left|\nabla^2 u^+\right|\right|_{L^2}+\left|\left|\nabla u^+\right|\right|_{L^2}\left|\left|\nabla(n^+, \theta)\right|\right|_{L^\infty}\left|\left|\nabla^2 u^+\right|\right|_{L^2}\\
&\quad+ \left|\left|\nabla u^+\right|\right|_{L^\infty}\left|\left|\nabla u^+\right|\right|_{L^2}\left|\left|\nabla^3 u^+\right|\right|_{L^2}+\left|\left|u^+ \right|\right|_{L^\infty}\left|\left|\nabla^2 u^+\right|\right|_{L^2}\left|\left|\nabla^3 u^+\right|\right|_{L^2}\\
&\quad+ \left|\left|\nabla \theta\right|\right|_{L^\infty}\left|\left|\nabla \theta\right|\right|_{L^2}\left|\left|\nabla^3 u^+\right|\right|_{L^2}+\left|\left|\theta \right|\right|_{L^\infty}\left|\left|\nabla^2 \theta\right|\right|_{L^2}\left|\left|\nabla^3 u^+\right|\right|_{L^2}\\
&\quad+ \left|\left| \theta\right|\right|_{L^\infty}\left|\left|\nabla^3 u^+\right|\right|_{L^2}\left|\left|\nabla^3 u^+\right|\right|_{L^2}+\left|\left|\nabla(n^+,\theta) \right|\right|_{L^\infty}\left|\left|\nabla^2 u^+\right|\right|_{L^2}\left|\left|\nabla^3 u^+\right|\right|_{L^2}\\
&\quad+ \left|\left| \nabla u^+\right|\right|_{L^6}\left|\left|\nabla^2(n^+,\theta) \right|\right|_{L^3}\left|\left|\nabla^3 u^+\right|\right|_{L^2}+\left|\left|\nabla u^+ \right|\right|_{L^2}\left|\left|\nabla(n^+,\theta)\right|\right|_{L^\infty}^2\left|\left|\nabla^3 u^+\right|\right|_{L^2}\\
&\lesssim \delta\left(\left|\left|\nabla\theta\right|\right|_{H^1}^2+\left|\left|\nabla u^+\right|\right|_{H^2}^2\right).
\end{split}\end{align}
Similarly,    we have
\begin{equation}\label{3.12}|J_3^k|\lesssim \delta\left(\left|\left|\nabla\theta\right|\right|_{H^1}^2+\left|\left|\nabla u^-\right|\right|_{H^2}^2\right).\end{equation}
Substituting \eqref{3.10}--\eqref{3.12} into \eqref{3.9}, then
summing $k$ from $1$ to $2$, we obtain \eqref{3.8} if $\delta$ is
small enough. The proof of Lemma \ref{Lemma3.2} is completed.
\end{proof}

In the following lemma, we give the energy estimate of the time
derivative for $(\theta,u^+,u^-)$.

\begin{Lemma}\label{Lemma3.3}Assume that the notations and hypotheses  of Theorem \ref{1mainth} and \eqref{3.1}  are in force, then
\begin{equation}\begin{split}&\frac{1}{2}\frac{d}{dt}
\left\{\frac{\beta_3}{\beta_1}\left|\left|\theta_t\right|\right|_{L^2}^2+
\left|\left|u^\pm_t\right|\right|_{L^2}^2\right\}+\left(\nu_1^\pm\left|\left|\nabla
u^\pm_t\right|\right|_{L^2}^2+\nu_2^\pm\left|\left|\text{\rm div}u^\pm_t\right|\right|_{L^2}^2\right)\\
&\quad~\leq C\delta(\|\nabla u^\pm(t)\|_{H^1}^2+\|\nabla
u^\pm_t(t)\|^2_{L^2}),\end{split}\label{3.13}
\end{equation}
and
\begin{equation}\label{3.14}\begin{split}&\frac{d}{dt}
\left\{\mu^\pm\left|\left|\nabla
u_{t}^\pm\right|\right|_{L^2}^2+(\mu^\pm+\lambda^\pm)
\left|\left|\text{\rm div} u_{t}^\pm\right|\right|_{L^2}^2\right\}+\left|\left|\sqrt{\rho^\pm} u_{tt}^\pm\right|\right|_{L^2}^2\\
&~\quad\le C\Big(\|\nabla u^\pm\|_{H^1}^2+\delta\|\nabla
u^\pm_{t}\|_{L^2}^2\Big).\end{split}
\end{equation}
\end{Lemma}
\begin{proof}Differentiating \eqref{2.1}$_{2,3,4}$ with respect to $t$, one has
\begin{equation}\label{3.15}
\left\{\begin{array}{l}
\partial_{tt}\theta+\beta_1\text{\rm div}u^+_t+\beta_2\text{\rm div}u^-_t=(F_1)_t,\\
\partial_{tt}u^{+}+\beta_3\nabla
\theta_t-\nu^+_1\Delta u^+_t-\nu^+_2\nabla\text{\rm div} u^+_t=(F_2)_t, \\
\partial_{tt}u^{-}+\beta_4\nabla \theta_t-\nu^-_1\Delta u^-_t-\nu^-_2\nabla\text{\rm div} u^-_t=(F_3)_t. \\
\end{array}\right.
\end{equation}
Multiplying \eqref{3.15}$_1$, \eqref{3.15}$_2$, \eqref{3.15}$_3$, by
$\frac{\beta_3}{\beta_1}\theta_t,u^+_t,u^-_t$ respectively, summing
up and then integrating the resulting equality over $\mathbb{R}^3$,
one has
\begin{equation}\begin{split}&
\frac{1}{2}\frac{d}{dt}\int_{\mathbb{R}^3} \frac{\beta_3}{\beta_1}
|\theta_t|^2+
|u^\pm_t|^2\mathrm{d}x+\int_{\mathbb{R}^3}\nu_1^{\pm}|\nabla
u_t^\pm(t)|^2\nu_2^\pm+\int_{\mathbb{R}^3}|\hbox{div} u^\pm_t(t)|^2\mathrm{d}x\\
=&~\int_{\mathbb{R}^3}\frac{\beta_3}{\beta_1}(F_1)_t\theta_t+(
F_2)_tu^+_t+( F_3)_tu^-_tdx.\end{split}\label{3.16}
\end{equation}
Using \eqref{1.15}, \eqref{2.1}$_{1,2}$ and the similar argument in
\eqref{3.10}, we obtain
\begin{equation}\label{3.17}\begin{split}&\left|\int_{\mathbb{R}^3} (F_1)_t\theta_t \mathrm{d}x\right|\\ \lesssim &~\|(F_1)_t\|_{L^2}\|\theta_t\|_{L^2}
\\ \lesssim&~\Big(\|n^\pm_t\|_{L^3}\|\nabla u^\pm\|_{L^6}+\|n^\pm\|_{L^\infty}\|\nabla u^\pm_t\|_{L^2}+\|\nabla n^\pm\|_{L^3}\| u^\pm_t\|_{L^6}+\|\nabla n^\pm_t\|_{L^2}\| u^\pm\|_{L^\infty}\\
&\quad +\|\theta_t\|_{L^2}\|\nabla u^\pm\|_{L^\infty}+\|\theta\|_{L^\infty}\|\nabla u^\pm_t\|_{L^2}\Big)\|\theta_t\|_{L^2}\\ \lesssim&~\Big(\|\nabla u^\pm\|_{L^3}\|\nabla u^\pm\|_{L^6}+\| u^\pm\|_{L^\infty}\|\nabla n^\pm\|_{L^3}\|\nabla u^\pm\|_{L^6}+\|n^\pm\|_{L^\infty}\|\nabla u^\pm_t\|_{L^2}+\|\nabla n^\pm\|_{L^3}\| u^\pm_t\|_{L^6}\\&~+\|\nabla^2 u^\pm\|_{L^2}\| u^\pm\|_{L^\infty}+\|(n^\pm,u^\pm)\|_{L^\infty}\|\nabla^2 (n^\pm,u^\pm)\|_{L^2}\| u^\pm\|_{L^\infty}\Big)\Big(\|\nabla u^\pm\|_{L^2}+\|F_1\|_{L^2}\Big)
\\ \lesssim&~\delta\left(\|\nabla u^\pm\|_{H^1}^2+\|\nabla u^\pm_t\|_{L^2}^2\right).
\end{split}\end{equation}
From integration by parts, we have
\begin{equation}\label{3.18}\begin{split}
\int_{\mathbb{R}^3}\left(\frac{\mu^{\pm}}{\rho^{\pm}}-\nu_1^\pm\right)\Delta
u^\pm_t
u^\pm_t\mathrm{d}x=-\int_{\mathbb{R}^3}\left(\frac{\mu^{\pm}}{\rho^{\pm}}-\nu_1^\pm\right)\nabla
u^\pm_t \nabla
u^\pm_t\mathrm{d}x-\int_{\mathbb{R}^3}\nabla\left(\frac{\mu^{\pm}}{\rho^{\pm}}-\nu_1^\pm\right)\nabla
u^\pm_t u^\pm_t\mathrm{d}x.\end{split}\end{equation} Performing the
similar procedure as \eqref{3.17}, we finally have
\begin{equation}\label{3.19}\left|\int_{\mathbb{R}^3} (F_2)_tu^+_t\mathrm{d}x\right|+\left|\int_{\mathbb{R}^3}
(F_3)_tu^-_t\mathrm{d}x\right|\lesssim \delta\left(\|\nabla
u^\pm\|_{H^1}+\|\nabla u^\pm_t\|_{L^2}\right).\end{equation}
Substituting \eqref{3.17} and \eqref{3.19} into \eqref{3.16}, we
obtain \eqref{3.13}.

\indent Next, we turn to the proof of \eqref{3.14}. To do this, we
first rewrite \eqref{3.15}$_2$ and \eqref{3.15}$_3$ in the following
form:
\begin{equation}\label{3.20}
\left\{\begin{array}{l}\rho^+u^{+}_{tt}-\mu^+\Delta u^+_t-(\mu^++\lambda^+)\nabla\operatorname{div} u^+_t\\
\quad=\nabla\theta_t-\rho^+_tu^{+}_{t}-\Big[\rho^+u^+\cdot\nabla
u^++\frac{\mu^+\left(\nabla u^{+}+\nabla^{t} u^{+}\right)
\nabla\alpha^+}{\alpha^+}+\frac{\lambda^+\operatorname{div} u^+\nabla\alpha^+}{\alpha^+}\Big]_t, \\
\rho^-u^{-}_{tt}-\mu^-\Delta u^-_t-(\mu^-+\lambda^-)\nabla\operatorname{div} u^-_t\\
\quad=\nabla\theta_t-\rho^-_tu^{-}_{t}-\Big[\rho^-u^-\cdot\nabla u^-+\frac{\mu^+\left(\nabla u^{-}+\nabla^{t} u^{-}\right)\nabla\alpha^-}{\alpha^-}+\frac{\lambda^-\operatorname{div} u^-\nabla\alpha^-}{\alpha^-}\Big]_t. \\
\end{array}\right.
\end{equation}
Multiplying \eqref{3.20}$_1$  and \eqref{3.20}$_2$ by $u^+_{tt}$ and $u^-_{tt}$ respectively, summing up and then integrating
the resulting equality over $\mathbb R^3$, one has
\begin{equation}\begin{split}&\frac{1}{2}\frac{d}{dt}\int_{\mathbb{R}^3}
 \mu^\pm| \nabla u_{t}^\pm|^2+(\mu^\pm+\lambda^\pm)| \text{div} u_{t}^\pm|^2\mathrm{d}x+\int_{\mathbb{R}^3}\rho^\pm |u_{tt}^\pm(t)|^2\mathrm{d}x\\
 &=\int_{\mathbb{R}^3} \nabla\theta_tu^\pm_{tt}\mathrm{d}x-\int_{\mathbb{R}^3}\rho^\pm_tu^{\pm}_{t}u^\pm_{tt}\mathrm{d}x\\
 &\quad-\int_{\mathbb{R}^3} \Big[\rho^\pm u^\pm\cdot\nabla u^\pm+\frac{\mu^+\left(\nabla u^{\pm}+\nabla^{t}
 u^{\pm}\right)\nabla\alpha^\pm}{\alpha^\pm}+\frac{\lambda^-\operatorname{div} u^\pm\nabla\alpha^\pm}{\alpha^\pm}\Big]_tu^\pm_{tt} \mathrm{d}x\\
 &\le C\left(\|\nabla\theta_t\|_{L^2}\|u^\pm_{tt}\|_{L^2}+\|\rho^\pm_t\|_{L^3}\|u^\pm_{t}\|_{L^6}\|u^\pm_{tt}\|_{L^2}
+\|u^\pm\|_{L^\infty}\|\nabla u^\pm_{t}\|_{L^2}\|u^\pm_{tt}\|_{L^2}\right.\\
&\quad+\|\nabla u^\pm\|_{L^3}\|
u^\pm_{t}\|_{L^6}\|u^\pm_{tt}\|_{L^2}+\|\nabla
\alpha^\pm\|_{L^\infty}\|\nabla u^\pm_{t}\|_{L^2}
\|u^\pm_{tt}\|_{L^2}+\|\nabla u^\pm\|_{L^\infty}\|\nabla
\alpha^\pm_{t}\|_{L^2}\|u^\pm_{tt}\|_{L^2}\\
&\quad+\left.\|\nabla u^\pm\|_{L^6}\| \|\nabla \alpha^\pm\|_{L^6}\| \alpha^\pm_{t}\|_{L^6}\|u^\pm_{tt}\|_{L^2}
+\|\rho^\pm_t\|_{L^2}\| u^\pm\|_{L^\infty}\| \nabla u^\pm\|_{L^\infty}\|u^\pm_{tt}\|_{L^2}\right)\\
&\le C\Big(\|\nabla u^\pm\|_{H^1}^2+\delta\|\nabla
u^\pm_{t}\|_{L^2}^2\Big)+\frac{1}{2}\int_{\mathbb{R}^3}\rho^\pm
|u_{tt}^\pm(t)|^2\mathrm{d}x,\end{split}\nonumber
\end{equation}
which shows \eqref{3.14}.
The proof of Lemma \ref{Lemma3.3} is completed.\end{proof}

Now, we are in a position to derive the time decay rates for $(\theta,u^+,u^-)$.

\begin{Theorem}\label{2mainth} Under the conditions of Theorem \ref{1mainth} and \eqref{3.1}, there exists a  positive constant $C$  such that
\begin{equation}\label{3.21}\begin{split}
\|(\theta,u^\pm)(t)\|_{L^2}\le CK_0(1+t)^{-\frac{3}{4}},
\end{split}\end{equation}
\begin{equation}\label{3.22}\|\nabla \theta(t)\|_{H^1}+\|\nabla u^\pm(t)\|_{H^2}+\|(\theta_t,u^\pm_t,\nabla u^\pm_t)\|^2_{L^2}\le CK_0(1+t)^{-\frac{5}{4}}
\end{equation}
and
\begin{equation}\label{3.23}\int_0^t(1+\tau)^{\frac{9}{8}}\| u_{tt}^\pm(\tau)\|_{L^2}^2\mathrm{d}\tau\le CK_0.
\end{equation}
\end{Theorem}
\begin{proof}From \eqref{2.1}, we see that

\begin{equation}\label{3.24-u}\|u^\pm_t\|_{L^2}\le C\left(\|(\nabla^2 u^\pm,\nabla \theta)\|_{L^2}+\delta\|\nabla u^\pm\|_{H^1}+\delta\|\nabla^2\theta\|_{L^2}\right),\end{equation}

\begin{equation}\label{3.24}\|\theta_t\|_{L^2}\le C\|\nabla u^\pm\|_{L^2},\end{equation}
and
\begin{equation}\nonumber\begin{split} \|\nabla^2\theta\|_{L^2}\le &~C\|(\nabla u^+_t,\nabla^3 u^+,\nabla F_2)\|_{L^2}\\
\le &~C\left(\|(\nabla u^+_t,\nabla^3 u^+)\|_{L^2}+\delta\|\nabla^2u^+\|_{H^1}+\delta\|\nabla^2\theta\|_{L^2}\right).
\end{split}\end{equation}
Therefore, we have
\begin{equation}\label{3.25}\begin{split} \|\nabla^2\theta\|_{L^2}
\le C\left(\|(\nabla u^+_t,\nabla^3 u^+)\|_{L^2}+\delta\|\nabla^2u^+\|_{L^2}\right).
\end{split}\end{equation}
Similarly, we also have
\begin{equation}\label{3.26}\begin{split} \|\nabla^3 u^\pm\|_{L^2}
\le C\left(\|(\nabla u^\pm_t,\nabla^2
\theta)\|_{L^2}+\delta\|\nabla^2u^\pm\|_{L^2}\right).
\end{split}\end{equation}
Let $D>0$ be a large but fixed constant. Using \eqref{3.24-u}--\eqref{3.26} and summing up
$D\times\left(\eqref{3.8}+\eqref{3.13}\right)+\eqref{3.14}$,
then there exists an energy functional $H(\theta,u^\pm)$ which is
equivalent to $\|\nabla\theta\|^2_{H^1}+\|\nabla
u^\pm\|^2_{H^2}+\|(\theta_t,u^\pm_t,\nabla u^\pm_t)\|^2_{L^2}$ such
that
\begin{equation}\label{3.27}\frac{d}{dt}H(t)+C_1\left(H(t)+\|u^\pm_{tt}\|_{L^2}^2\right)\le C\|\nabla(\theta^l,u^{\pm,l})\|_{L^2}^2.
\end{equation}
for some positive constant $C_1$.
We also define the time--weighted energy functional
\begin{equation}\label{3.28} \mathcal E(t)=\sup\limits_{0\leq\tau\leq t}\big
\{(1+\tau)^{\frac{5}{2}}H(\tau)\big \}.
\end{equation}
It is clear that $\mathcal E(t)$  is non--decreasing. Denoting
$$\mathcal F=(F^1,F^2,F^3)^t,$$ we have from Duhamel principle
that
\begin{equation}\nonumber U^l=\text{e}^{t\mathcal{B}}U^l(0)+\int_0^t\text{e}^{(t-\tau)\mathcal{B}}\mathcal
F^l(\tau)\mathrm{d}\tau,
\end{equation}
which together with Proposition \ref{Prop2.6},  Plancherel theorem,
H\"older inequality and Hausdorff--Young
 inequality leads to
\begin{equation}\label{3.29}\begin{split}&\|\nabla(\theta^l,u^{\pm,l})(t)\|_{L^{2}}\\
 \lesssim &~(1+t)^{-\frac{5}{4}}\|  (\theta,u^{\pm})(0)\|_{L^1}+
 \int_0^t\left\|\nabla \text{e}^{(t-\tau)\mathcal{B}}\mathcal F^l(\tau)\right\|_{L^2}\mathrm{d}\tau.
\end{split}\end{equation}
Due to non--dissipation property of the variable $n^\pm$, it
requires us to develop new thoughts to deal with the last two terms
of $F_1$ on the right-hand side of $\eqref{2.1}_2$. The main idea
here is that we consider the two trouble terms as one: $-\mathcal
C\rho^-u^+\cdot\nabla n^+-\mathcal C\rho^+u^-\cdot\nabla n^-$, and
then rewrite it in a clever way. To see this,  by noticing that
$\rho^-s_-^2-\rho^+s_+^2\neq 0$, and fully using the subtle relation
between the variables, we surprisingly find that the trouble term
$-\mathcal C\rho^-u^+\cdot\nabla n^+-\mathcal C\rho^+u^-\cdot\nabla
n^-$ can be rewritten as
\begin{equation}\begin{split}\label{3.6-1}&-\mathcal C\rho^-u^+\cdot\nabla n^+-\mathcal C\rho^+u^-\cdot\nabla n^-
\\=&-\mathcal C\rho^-u^+\cdot\nabla n^+-\mathcal C\rho^+u^-\cdot\left(\frac{\nabla\theta}{\mathcal C\rho^+}-\frac{\rho^-\nabla n^+}{\rho^+}\right)
\\=&-u^-\cdot\nabla\theta+\mathcal C\rho^-(u^--u^+)\cdot\nabla n^+
\\=&-u^-\cdot\nabla\theta+\frac{s_+^2s_-^2\rho^+\rho^-(u^--u^+)}{(\rho^+s_+)^2+(\rho^-s_-^2-\rho^+s_+^2)R^+}\cdot\nabla n^+
\\=&-u^-\cdot\nabla\theta+\frac{s_+^2s_-^2\rho^+\rho^-(u^--u^+)}{\rho^-s_-^2-\rho^+s_+^2}\cdot\nabla \ln \left((\rho^+s_+)^2+(\rho^-s_-^2-\rho^+s_+^2)R^+\right)
\\&+\frac{s_+^2s_-^2\rho^+\rho^-(u^+-u^-)}{\rho^-s_-^2-\rho^+s_+^2}\cdot\left(\nabla (\rho^+s_+)^2+R^+\nabla(\rho^-s_-^2-\rho^+s_+^2)\right)
\\=&-u^-\cdot\nabla\theta+\text{div}\left[\frac{s_+^2s_-^2\rho^+\rho^-(u^--u^+)}{\rho^-s_-^2-\rho^+s_+^2}
\ln\left(\frac{ (\rho^+s_+)^2+(\rho^-s_-^2-\rho^+s_+^2)R^+}{(\bar\rho^+\bar s_+)^2+(\bar\rho^-\bar s_-^2-\bar\rho^+\bar s_+^2)\bar R^+}\right)
\right]
\\&+\ln\left(\frac{ (\rho^+s_+)^2+(\rho^-s_-^2-\rho^+s_+^2)R^+}{(\bar\rho^+\bar s_+)^2+(\bar\rho^-\bar s_-^2-
\bar\rho^+\bar s_+^2)\bar R^+}\right)\text{div}\left[\frac{s_+^2s_-^2\rho^+\rho^-(u^+-u^-)}{\rho^-s_-^2-\rho^+s_+^2} \right]
\\&+\frac{s_+^2s_-^2\rho^+\rho^-(u^+-u^-)}{\rho^-s_-^2-\rho^+s_+^2}\cdot\left(\nabla (\rho^+s_+)^2+R^+\nabla(\rho^-s_-^2-\rho^+s_+^2)\right)
\\=&-u^-\cdot\nabla\theta+\ln\left(\frac{ (\rho^+s_+)^2+(\rho^-s_-^2-\rho^+s_+^2)R^+}{(\bar\rho^+\bar s_+)^2+(\bar\rho^-\bar s_-^2-\bar\rho^+\bar s_+^2)\bar R^+}\right)\text{div}\left[\frac{s_+^2s_-^2\rho^+\rho^-(u^+-u^-)}{\rho^-s_-^2-\rho^+s_+^2} \right]
\\&+\frac{s_+^2s_-^2\rho^+\rho^-(u^+-u^-)}{\rho^-s_-^2-\rho^+s_+^2}\cdot\left(\nabla (\rho^+s_+)^2+R^+\nabla(\rho^-s_-^2-\rho^+s_+^2)\right)
+\text{div}F_{11}.
\end{split}\end{equation}
 Setting $\mathcal F_1=(\text{div}F_{11},0,0)^t$ and $\mathcal F_2=\mathcal F-\mathcal F_1$, then from Proposition \ref{Prop2.6},
 \eqref{3.28} and \eqref{3.29}, we have
 \begin{equation}\label{3.31}\begin{split}&\|\nabla(\theta^l,u^{\pm,l})(t)\|_{L^{2}}\\
 \lesssim &~(1+t)^{-\frac{5}{4}}\|  (\theta,u^{\pm})(0)\|_{L^1}+\int_0^t(1+t-\tau)^{-\frac{5}{4}}\left\|\mathcal F_2(\tau)\right\|_{L^1}\mathrm{d}\tau\\&
 +\int_0^t(1+t-\tau)^{-\frac{5}{4}}\|F_{11}\|_{L^\frac{3}{2}}\mathrm{d}\tau\\
 \lesssim &~(1+t)^{-\frac{5}{4}}K_0+\int_0^t(1+t-\tau)^{-\frac{5}{4}}
 \|(n^\pm,\theta,u^\pm,\nabla n^\pm,\nabla \theta)
 (\tau)\|_{L^2}\|(\nabla\theta,\nabla u^\pm,\nabla^2u^\pm)(\tau)\|_{L^2}\mathrm{d}\tau
 \\&+\int_0^t(1+t-\tau)^{-\frac{5}{4}}\|(n^\pm,\theta)(\tau)\|_{L^2}\|u^\pm(\tau)\|_{L^6}\mathrm{d}\tau
 \\
 \lesssim &~(1+t)^{-\frac{5}{4}}K_0+\delta\int_0^t(1+t-\tau)^{-\frac{5}{4}}(1+\tau)^{-\frac{5}{4}}\sqrt{\mathcal E(t)}\mathrm{d}\tau
 \\
 \lesssim &~(1+t)^{-\frac{5}{4}}K_0+\delta(1+t)^{-\frac{5}{4}}\sqrt{\mathcal E(t)}.
\end{split}\end{equation}
Substituting \eqref{3.31} into \eqref{3.27}, we have
\begin{equation}\label{3.32}\frac{d}{dt}H(t)+C_1\left(H(t)+\|u^\pm_{tt}\|_{L^2}^2\right)\le C(1+t)^{-\frac{5}{2}}K^2_0+C\delta^2(1+t)^{-\frac{5}{2}}\mathcal E(t).
\end{equation}
Applying Gronwall inequality to \eqref{3.32}, we have
\begin{equation}\nonumber H(t)\le C(1+t)^{-\frac{5}{2}}K^2_0+C\delta^2(1+t)^{-\frac{5}{2}}\mathcal E(t),
\end{equation}
that is,
\begin{equation}\nonumber \mathcal E(t)\le CK^2_0+C\delta^2\mathcal E(t),
\end{equation}
which shows \eqref{3.22} if $\delta$ is small enough.

Employing the similar argument as in \eqref{3.31}, we have
\begin{equation}\label{3.33}\begin{split}&\|(\theta^l,u^{\pm,l})(t)\|_{L^{2}}\\
 \lesssim &~(1+t)^{-\frac{3}{4}}\|  (\theta,u^{\pm})(0)\|_{L^1}+\int_0^t(1+t-\tau)^{-\frac{3}{4}}\left\|\mathcal F_2(\tau)
 \right\|_{L^1}\mathrm{d}\tau\\&
 +\int_0^t(1+t-\tau)^{-\frac{3}{4}}\|F_{11}\|_{L^\frac{3}{2}}\mathrm{d}\tau\\
 \lesssim &~(1+t)^{-\frac{3}{4}}K_0+\int_0^t(1+t-\tau)^{-\frac{5}{4}}\|(n^\pm,u^\pm,\nabla n^\pm)(\tau)
 \|_{L^2}\|(\nabla\theta,\nabla u^\pm,\nabla^2u^\pm)(\tau)\|_{L^2}\mathrm{d}\tau\\
 &+\int_0^t(1+t-\tau)^{-\frac{3}{4}}\|n^\pm(\tau)\|_{L^2}\|u^\pm(\tau)\|_{L^6}\mathrm{d}\tau
 \\
 \lesssim &~(1+t)^{-\frac{3}{4}}K_0+\delta^2\int_0^t(1+t-\tau)^{-\frac{3}{4}}(1+\tau)^{-\frac{5}{4}}\mathrm{d}\tau
 \\
 \lesssim &~(1+t)^{-\frac{3}{4}}K_0+\delta^2(1+t)^{-\frac{3}{4}}.
\end{split}\end{equation}
Substituting \eqref{3.33} into \eqref{3.3}, there exists a positive constant $C_2$ such that
\begin{equation}\nonumber\begin{split}&\frac{\rm{d}}{{\rm{d}}t}\left\{\left|\left|\theta\right|\right|^2_{L^2}+\frac{\beta_1}{\beta_3}\left|\left|u^+\right|\right|^2_{L^2}
+\frac{\beta_2}{\beta_4}\left|\left|u^-\right|\right|^2_{L^2}\right\}\\
&\quad+C_2\left(\left|\left|\theta\right|\right|^2_{L^2}+\frac{\beta_1}{\beta_3}\left|\left|u^+\right|\right|^2_{L^2}
+\frac{\beta_2}{\beta_4}\left|\left|u^-\right|\right|^2_{L^2}\right)\\
&\lesssim\left|\left|(\theta^l,u^{\pm,l})\right|\right|^2_{L^2}+\left|\left|\nabla \theta\right|\right|^2_{L^2}\\
&\lesssim(1+t)^{-\frac{3}{2}}K_0^2.
\end{split}\end{equation}
Applying Gronwall inequality to the above inequality, we obtain
\eqref{3.21}.

Multiplying \eqref{3.32} by $(1+t)^\frac{9}{8}$, we have
\begin{equation}\label{3.34}\begin{split}&\frac{d}{dt}(1+t)^\frac{9}{8}H(t)+C_1(1+t)^\frac{9}{8}
\left(H(t)+\|u^\pm_{tt}\|_{L^2}^2\right)\\
&\le
C\left((1+t)^{\frac{9}{8}}(1+t)^{-\frac{5}{2}}K^2_0+(1+t)^{\frac{1}{8}}H(t)\right)\\
&\le
C\left((1+t)^{\frac{9}{8}}(1+t)^{-\frac{5}{2}}K^2_0+(1+t)^{\frac{1}{8}}(1+t)^{-\frac{5}{2}}K^2_0\right)\\
&\le C(1+t)^{-\frac{11}{8}}K^2_0.
\end{split}\end{equation}
Integrating \eqref{3.34} from $0$ to $t$, we obtain \eqref{3.23}. The proof of Theorem \ref{2mainth} is completed.
\end{proof}

Now, we are in a position to establish the a priori estimate for $\left(n^+,n^-\right)$.

\begin{Lemma}\label{Lemma3.5}Assume that the notations and hypotheses  of Theorem \ref{1mainth} and \eqref{3.1}  are in force, then
\begin{equation}\begin{split}\|(n^+,n^-)\|_{H^3}\le CK_0.\end{split}\label{3.35}
\end{equation}
\end{Lemma}
\begin{proof}As already stated, we cannot work on the linearize system \eqref{1.25}.
Instead, we introduce the following new semi--linearized system of
model \eqref{1.5}:
\begin{equation}\label{3.36}
\left\{\begin{array}{l}
\partial_{t} n^++(1+n^+)\operatorname{div}u^++u^+\cdot\nabla n^+=0, \\
\rho^+\partial_{t}u^{+}+\mathcal{C}\left(\rho^{-} \nabla
n^{+}+\rho^{+} \nabla n^{-}\right)-\mu^+\Delta u^+-(\mu^++\lambda^+)\nabla\operatorname{div} u^+=G_1, \\
\partial_{t} n^-+(1+n^-)\operatorname{div}u^-+u^-\cdot\nabla n^-=0, \\
\rho^-\partial_{t}u^{-}+\mathcal{C}\left(\rho^{-} \nabla
n^{+}+\rho^{+} \nabla n^{-}\right)-\mu^-\Delta u^--(\mu^-+\lambda^-)\nabla\operatorname{div} u^-=G_2, \\
\end{array}\right.
\end{equation}
where $$G_1=-\rho^+ u^+\cdot\nabla u^++\frac{\mu^+\left(\nabla u^{+}+\nabla^{t} u^{+}\right)\nabla\alpha^+}{\alpha^+}+\frac{\lambda^+\operatorname{div} u^+\nabla\alpha^+}{\alpha^+}$$
and
$$G_2=-\rho^- u^-\cdot\nabla u^-+\frac{\mu^-\left(\nabla u^{-}+\nabla^{t} u^{-}\right)\nabla\alpha^-}{\alpha^-}+\frac{\lambda^+\operatorname{div} u^-\nabla\alpha^-}{\alpha^-}.$$

For $0\le \ell\le 2$, multiplying $\nabla^{\ell}\eqref{3.36}_{1}$ and $\nabla^{\ell}\eqref{3.36}_{3}$  by $\nabla^{\ell}n^+$ and $\nabla^{\ell}n^-$, and then integrating over $\mathbb R^3,$ we obtain
\begin{equation}\label{3.37}\begin{split}&\frac{1}{2}\frac{d}{dt}\int_{\mathbb{R}^3}|\nabla^\ell n^\pm|^2\mathrm{d}x
=-\int_{\mathbb{R}^3}\nabla^\ell\left(\text{div}u^\pm+n^\pm\operatorname{div}u^\pm+u^\pm\cdot\nabla
n^\pm\right)\nabla^\ell n^\pm \mathrm{d}x.\end{split}
\end{equation}
Integrating by parts, we see that
\begin{equation}\nonumber\int_{\mathbb{R}^3} u^\pm\cdot\nabla\nabla^\ell n^\pm\nabla^\ell n^\pm
\mathrm{d}x=-\frac{1}{2}\int_{\mathbb{R}^3}
\text{div}u^\pm|\nabla^\ell n^\pm|^2\mathrm{d}x.
\end{equation}
Then using Lemmas \ref{1interpolation}--\ref{es-product} and
H\"older inequality, we can deal with \eqref{3.37} as follows
\begin{equation}\begin{split}&\frac{d}{dt}\int_{\mathbb{R}^3}|\nabla^\ell n^\pm|^2\mathrm{d}x\le C\|\nabla u^\pm\|_{H^2}\|n^\pm\|_{H^2}.\nonumber\end{split}
\end{equation}
Therefore, we have
\begin{equation}\label{3.38}\begin{split}&\frac{d}{dt}\|n^\pm\|_{H^2}\le C\|\nabla u^\pm\|_{H^2}.\end{split}
\end{equation}
In virtue of Theorem \ref{2mainth}, we see that
$\displaystyle\int_0^t\|\nabla u^\pm(\tau)\|_{H^2}\mathrm{d}\tau\le
CK_0$. Then integrating the above inequality from $0$ to $t$, we see
that $\|n^\pm\|_{H^2}\le CK_0$.\par However, since we don't know
whether the $L^1(0,t)$--norm of $\|\nabla^{4}u^\pm\|$ is uniformly
bounded or not, it seems impossible to derive the uniform estimates
on $\|\nabla^3 n^{\pm}\|_{L^2}$ from \eqref{1.27}. Therefore, we
must pursue another route by resorting to the momentum equations in
\eqref{3.36}. To begin with, by applying the classic $L^p$--estimate
of elliptic system to \eqref{3.20}, we have
\begin{equation}\label{3.39}\begin{split}\|\nabla^2 u^\pm_t\|_{L^2}&\lesssim
\|(u^\pm_{tt},\nabla\theta_t)\|_{L^2}+\|(\rho^\pm_t,\nabla u^\pm)\|_{L^\infty}\|(u^\pm_t,\nabla n^\pm_t)\|_{L^2}+\|( u^\pm,\nabla n^\pm)\|_{L^\infty}\|\nabla u^\pm_t\|_{L^2}\\
&\quad +\|\rho^\pm_t\|_{L^\infty}\|u^\pm\|_{L^\infty}\|\nabla
u^\pm\|_{L^2} +\|(n^\pm_t,\theta_t)\|_{L^2}\|\nabla
u^\pm\|_{L^\infty}\|(\nabla n^\pm,\nabla \theta)\|_{L^\infty}\\
&\lesssim \|(u^\pm_{tt},u^\pm_t,\nabla u^\pm_t)\|_{L^2}+\|\nabla
u^\pm\|_{H^1}.\end{split}
\end{equation}
Summing up  $\text{div}\eqref{3.36}_2-\text{div}\eqref{3.36}_4$, we obtain
\begin{equation}\label{3.40}\begin{split}\text{div}(\rho^+u^+_t-\rho^-u^-_t)-\Delta[(2\mu^++\lambda^+)\text{div}u^+-(2\mu^-+\lambda^-)\text{div}u^-]=\text{div}(G_1-G_2).\end{split}
\end{equation}
Denote the linear combination of the two velocities by
$v=(2\mu^++\lambda^+)u^+-(2\mu^-+\lambda^-)u^-$. Applying standard
$L^p$--estimate of elliptic system to \eqref{3.40}, we can obtain
the following key estimate of $v$
\begin{equation}\label{3.41}\begin{split}\|\nabla^3\text{div} v\|_{L^2}\lesssim&~\|\nabla^2(\rho^\pm u^\pm_{t},G_1-G_2)\|_{L^2}\\
\lesssim&~\|\rho^\pm\|_{W^{1,\infty}} \|\nabla
u^\pm_{t}\|_{H^1}+\|\nabla^2\rho^\pm\|_{L^2}
\|u^\pm_{t}\|_{L^6}+\|\nabla^2(\nabla u^\pm  \nabla
\theta)\|_{L^2}\\&+\|\nabla u^\pm\|_{L^\infty}\|(u^\pm,\nabla
n^\pm)\|_{H^2}+\|(u^\pm,\nabla n^\pm)\|_{L^\infty}\|\nabla
u^\pm\|_{H^2}\\ \lesssim&~ \|\nabla u^\pm_{t}\|_{H^1}+\|\nabla
u^\pm\|_{H^2}.\end{split}
\end{equation}
\noindent Multiplying $\nabla^{3}\eqref{3.36}_{2}$ and
$\nabla^{3}\eqref{3.36}_{4}$  by $\nabla^{3}u^+$ and
$\nabla^{3}u^-$, and then integrating over $\mathbb R^3$, we obtain
\begin{equation}\label{3.42}\begin{split}&\frac{1}{2}\frac{d}{dt}\int_{\mathbb{R}^3}\rho^\pm|\nabla^3 u^\pm|^2\mathrm{d}x+\mu^\pm
\int_{\mathbb{R}^3}|\nabla^4u^\pm|^2\mathrm{d}x+(\mu^\pm+\nu^\pm)\int_{\mathbb{R}^3}|\nabla^3\text{div}u^\pm|^2\mathrm{d}x\\
=&~\frac{1}{2}\int_{\mathbb{R}^3}\rho^\pm_t|\nabla^3
u^\pm|^2\mathrm{d}x-\sum\limits_{m=0}^2\int_{\mathbb{R}^3}\nabla^{3-m}\rho^\pm\nabla^mu^\pm_t\nabla^3
u^\pm
\mathrm{d}x\\
&\quad-\int_{\mathbb{R}^3}\nabla^3\left[\mathcal{C}\left(\rho^{-}
\nabla n^{+}+\rho^{+} \nabla n^{-}\right)\right]\nabla^3 u^+
\mathrm{d}x\\&-\int_{\mathbb{R}^3}\nabla^3\left[\mathcal{C}\left(\rho^{-}
\nabla n^{+}+\rho^{+} \nabla n^{-}\right)\right]\nabla^3 u^-
\mathrm{d}x+\int_{\mathbb{R}^3}\nabla^3G_1\nabla^3 u^+
\mathrm{d}x+\int_{\mathbb{R}^3}\nabla^3G_2\nabla^3 u^-
\mathrm{d}x\\=&~\sum\limits_{i=1}^6I_i.\end{split}
\end{equation}
From \eqref{2.1}$_1$, we have
\begin{equation}\label{3.43}|I_1|\lesssim \|\rho^\pm_t\|_{L^\infty}\|\nabla^3u^\pm\|_{L^2}^2\lesssim \|u^\pm\|_{W{1,\infty}}\|\nabla^3u^\pm\|^2_{L^2}\lesssim \delta\|\nabla^3u^\pm\|_{L^2}^2
\end{equation}
and from \eqref{3.39}, we also have
\begin{equation}\label{3.44}\begin{split}|I_2|\lesssim &\Big(\|\nabla\rho^\pm\|_{L^\infty}\|\nabla^2u^\pm_{t}\|_{L^2}
+\|\nabla^2\rho^\pm\|_{L^3}\|\nabla
u^\pm_{t}\|_{L^6}+\|\nabla^3\rho^\pm\|_{L^2}\|
u^\pm_{t}\|_{L^\infty}\Big)\|\nabla^3u^\pm\|_{L^2}\\ \lesssim
&\delta\Big(\|(u^\pm_{tt},u^\pm_t,\nabla u^\pm_t)\|^2_{L^2}+\|\nabla
u^\pm\|_{H^2}\Big)
\end{split}\end{equation}
By virtue of $\eqref{3.36}_{1}$, $\eqref{3.36}_{3}$, \eqref{3.39}
and \eqref{3.41}, we get
\begin{align}I_3=&-\int_{\mathbb{R}^3}\nabla^3\left[\mathcal{C}\left(\rho^{-} \nabla
n^{+}+\rho^{+} \nabla n^{-}\right)\right]\nabla^3
u^+\mathrm{d}x\nonumber\\=&-\int_{\mathbb{R}^3}\left[\mathcal{C}\left(\rho^{-}
\nabla\nabla^3 n^{+}+\rho^{+} \nabla
\nabla^3n^{-}\right)\right]\nabla^3 u^+
\mathrm{d}x-\sum\limits_{m=0}^2\int_{\mathbb{R}^3}\nabla^{3-m}(\mathcal{C}\rho^{-})
\nabla\nabla^m
n^{+}\nabla^3u^+\mathrm{d}x\nonumber\\&-\sum\limits_{m=0}^2\int_{\mathbb{R}^3}\nabla^{3-m}(\mathcal{C}\rho^{+})
\nabla\nabla^m n^{-}\nabla^3u^+\mathrm{d}x\nonumber\\ =
&\int_{\mathbb{R}^3}\mathcal{C}\rho^{-} \nabla^3 n^{+}\nabla^3
\text{div}u^+\mathrm{d}x+\int_{\mathbb{R}^3}\mathcal{C}\rho^{+}
\nabla^3 n^{-}\nabla^3
\text{div}u^+\mathrm{d}x\nonumber\\
&+\int_{\mathbb{R}^3}\left[\nabla(\mathcal{C}\rho^{-}) \nabla^3
n^{+}+\nabla(\mathcal{C}\rho^{+}) \nabla^3n^{-}\right]\nabla^3 u^+
\mathrm{d}x\nonumber\\&-\sum\limits_{m=0}^2\int_{\mathbb{R}^3}t\nabla^{3-m}(\mathcal{C}\rho^{-})
\nabla\nabla^m
n^{+}\nabla^3u^+\mathrm{d}x-\sum\limits_{m=0}^2\int_{\mathbb{R}^3}\nabla^{3-m}(\mathcal{C}\rho^{+})
\nabla\nabla^m n^{-}\nabla^3u^+\mathrm{d}x\nonumber\\ =
&-\int_{\mathbb{R}^3}\mathcal{C}\rho^{-} \nabla^3 n^{+}\nabla^3
\left(\frac{n^+_t+u^+\cdot\nabla
n^+}{1+n^+}\right)\mathrm{d}x\nonumber\\
&+\int_{\mathbb{R}^3}\mathcal{C}\rho^{+} \nabla^3 n^{-}\nabla^3
\left(\frac{\text{div}v+(2\mu^-+\lambda^-)\text{div}u^-}{2\mu^++\lambda^+}\right)\mathrm{d}x\nonumber\\
&+\int_{\mathbb{R}^3}\left[\nabla(\mathcal{C}\rho^{-}) \nabla^3
n^{+}+\nabla(\mathcal{C}\rho^{+}) \nabla^3n^{-}\right]\nabla^3 u^+
\mathrm{d}x-\sum\limits_{m=0}^2\int_{\mathbb{R}^3}\nabla^{3-m}(\mathcal{C}\rho^{-})
\nabla\nabla^m
n^{+}\nabla^3u^+\mathrm{d}x\nonumber\\&-\sum\limits_{m=0}^2\int_{\mathbb{R}^3}\nabla^{3-m}(\mathcal{C}\rho^{+})
\nabla\nabla^m n^{-}\nabla^3u^+\mathrm{d}x\nonumber\\ =
&-\frac{1}{2}\frac{d}{dt}\int_{\mathbb{R}^3}
\frac{\mathcal{C}\rho^{-} |\nabla^3
n^{+}|^2}{1+n^+}\mathrm{d}x+\frac{1}{2}\int_{\mathbb{R}^3} |\nabla^3
n^{+}|^2\left(\frac{\mathcal{C}\rho^{-}
}{1+n^+}\right)_tdx\nonumber\\
&-\sum\limits_{m=0}^2\int_{\mathbb{R}^3}\mathcal{C}\rho^{-} \nabla^3
n^{+}\nabla^mn^+_t\nabla^{3-m}
\left(\frac{1}{1+n^+}\right)\mathrm{d}x\nonumber\\&+\frac{1}{2}
\int_{\mathbb{R}^3}|\nabla^3n^+|^2\text{div}\left(\frac{\mathcal{C}\rho^{-}u^+}{1+n^+}\right)\mathrm{d}x
-\sum\limits_{m=0}^2\int_{\mathbb{R}^3}\mathcal{C}\rho^{-}\nabla^3n^+\nabla^{3-m}\left(\frac{u^+}{1+n^+}\right)\cdot\nabla\nabla^m
n^+\mathrm{d}x
\nonumber\\&+\frac{2\mu^-+\lambda^-}{2\mu^++\lambda^+}\Big(-\frac{1}{2}\frac{d}{dt}\int_{\mathbb{R}^3}
\frac{\mathcal{C}\rho^{+} |\nabla^3
n^{-}|^2}{1+n^-}\mathrm{d}x+\frac{1}{2}\int_{\mathbb{R}^3} |\nabla^3
n^{-}|^2\left(\frac{\mathcal{C}\rho^{+} }{1+n^-}\right)_tdx
\nonumber\\&-\sum\limits_{m=0}^2\int_{\mathbb{R}^3}\mathcal{C}\rho^{+}
\nabla^3 n^{-}\nabla^mn^-_t\nabla^{3-m}
\left(\frac{1}{1+n^-}\right)\mathrm{d}x+\frac{1}{2}\int_{\mathbb{R}^3}
|\nabla^3n^-|^2\text{div}\left(\frac{\mathcal{C}\rho^{+}u^-}{1+n^-}\right)\mathrm{d}x
\nonumber\\&-\sum\limits_{m=0}^2\int_{\mathbb{R}^3}\mathcal{C}\rho^{+}\nabla^3n^-\nabla^{3-m}\left(\frac{u^-}{1+n^-}\right)
\cdot\nabla\nabla^m n^-\mathrm{d}x\Big)
+\frac{1}{2\mu^++\lambda^+}\int_{\mathbb{R}^3}\mathcal{C}\rho^{+}\nabla^3n^-\nabla^3\text{div}v\mathrm{d}x
\nonumber\\&+\int_{\mathbb{R}^3}\left[\nabla(\mathcal{C}\rho^{-})
\nabla^3 n^{+}+\nabla(\mathcal{C}\rho^{+})
\nabla^3n^{-}\right]\nabla^3 u^+
\mathrm{d}x-\sum\limits_{m=0}^2\int_{\mathbb{R}^3}\nabla^{3-m}(\mathcal{C}\rho^{-})
\nabla\nabla^m
n^{+}\nabla^3u^+\mathrm{d}x\nonumber\\&-\sum\limits_{m=0}^2\int_{\mathbb{R}^3}\nabla^{3-m}(\mathcal{C}\rho^{+})
\nabla\nabla^m n^{-}\nabla^3u^+\mathrm{d}x \nonumber\\ \le
&-\frac{1}{2}\frac{d}{dt}\int_{\mathbb{R}^3}
\frac{\mathcal{C}\rho^{-} |\nabla^3
n^{+}|^2}{1+n^+}\mathrm{d}x-\frac{2\mu^-+\lambda^-}{2(2\mu^++\lambda^+)}\frac{d}{dt}\int_{\mathbb{R}^3}
\frac{\mathcal{C}\rho^{+} |\nabla^3
n^{-}|^2}{1+n^-}\mathrm{d}x\nonumber\\
&+C\delta\Big(\|(u^\pm_{tt},u^\pm_t,\nabla u^\pm_t)\|_{L^2}+\|\nabla
u^\pm\|_{H^2}\Big).\label{3.45}
\end{align}
Employing the similar argument as in \eqref{3.45}, we obtain
\begin{equation}\begin{split}\label{3.46}I_4&\le -\frac{1}{2}\frac{d}{dt}\int_{\mathbb{R}^3} \frac{\mathcal{C}\rho^{+} |\nabla^3
n^{-}|^2}{1+n^-}\mathrm{d}x-\frac{2\mu^++\lambda^+}{2(2\mu^-+\lambda^-)}\frac{d}{dt}\int_{\mathbb{R}^3}
\frac{\mathcal{C}\rho^{-} |\nabla^3
n^{+}|^2}{1+n^+}\mathrm{d}x\\
&\quad+C\delta\Big(\|(u^\pm_{tt},u^\pm_t,\nabla
u^\pm_t)\|_{L^2}+\|\nabla u^\pm\|_{H^2}\Big)
\end{split}\end{equation} By applying integration by parts, we have
\begin{equation}\begin{split}\label{3.47}I_5&=\displaystyle-\int_{\mathbb{R}^3}\nabla^2G_1\nabla^4u^+\mathrm{d}x\\
&\le C\|\nabla^2G_1\|_{L^2}\|\nabla^4u^+\|_{L^2}\\& \le
C\Big(\|(u^\pm,\nabla n^\pm)\|_{H^2}\|\nabla
u^\pm\|_{L^\infty}+\|(u^\pm,\nabla n^\pm)\|_{L^\infty}\|\nabla
u^\pm\|_{H^2}\Big)\|\nabla^4u^+\|_{L^2}\\& \le C\delta\Big(\|\nabla
u^\pm\|_{H^2}^2+\|\nabla^4u^+\|_{L^2}^2\Big).
\end{split}\end{equation}
Similarly, for $I_6$, we have
\begin{equation}\begin{split}\label{3.48}I_6 \le&~ C\delta(\|\nabla u^\pm\|_{H^2}^2+\|\nabla^4u^-\|_{L^2}^2).
\end{split}\end{equation}
Substituting \eqref{3.43}--\eqref{3.48} into \eqref{3.42} yields
\begin{equation}\label{3.49}\begin{split}&\frac{d}{dt}\int_{\mathbb{R}^3}\rho^\pm|\nabla^3 u^\pm|^2+\frac{\mathcal{C}\rho^{-} |\nabla^3
n^{+}|^2}{1+n^+}+\frac{2\mu^-+\lambda^-}{2\mu^++\lambda^+}\frac{\mathcal{C}\rho^{+}
|\nabla^3 n^{-}|^2}{1+n^-}\\
&\quad+\frac{\mathcal{C}\rho^{+} |\nabla^3
n^{-}|^2}{1+n^-}+\frac{2\mu^++\lambda^+}{2\mu^-+\lambda^-}\frac{\mathcal{C}\rho^{-}
|\nabla^3
n^{+}|^2}{1+n^+}\mathrm{d}x\\&\quad+\mu^\pm\int_{\mathbb{R}^3}|\nabla^4u^\pm|^2\mathrm{d}x+
(\mu^\pm+\nu^\pm)\int_{\mathbb{R}^3}|\nabla^3\text{div}u^\pm|^2\mathrm{d}x\\
&\le C\delta\Big(\|(u^\pm_{tt},u^\pm_t,\nabla
u^\pm_t)\|_{L^2}+\|\nabla u^\pm\|_{H^2}\Big).\end{split}
\end{equation}
In virtue of  \eqref{3.22} and \eqref{3.23}, for any $t\ge 0$, it
holds that
$$\int_0^t\left(\|(u^\pm_{tt},u^\pm_t,\nabla u^\pm_t)(\tau)\|_{L^2}+\|\nabla u^\pm(\tau)\|_{H^2}\right)\mathrm{d}\tau \le CK_0.$$
Consequently, integrating \eqref{3.49} from $0$ to $t$, we finally
deduce that
\begin{equation}\label{3.50}\|\nabla^3n^\pm\|_{L^2}\le CK_0.\end{equation}
\par
\noindent Therefore, we complete the proof of Lemma \ref{Lemma3.5}.
\end{proof}

\bigskip

\noindent{\bf Proof of Theorem \ref{1mainth}. } Using Theorem \ref{2mainth} and Lemma \ref{Lemma3.5}, we can obtain Theorem \ref{1mainth} immediately.

\bigskip
\appendix

\section{Analytic tools}\label{1section_appendix}
We recall the Sobolev interpolation of the Gagliardo--Nirenberg inequality.
 \begin{Lemma}\label{1interpolation}
 Let $0\le i, j\le k$, then we have
\begin{equation}\nonumber
\norm{\nabla^i f}_{L^p}\lesssim \norm{  \nabla^jf}_{L^q}^{1-a}\norm{ \nabla^k f}_{L^r}^a
\end{equation}
where $a$ belongs to $[\frac{i}{k},1]$ and satisfies
\begin{equation}\nonumber
\frac{i}{3}-\frac{1}{p}=\left(\frac{j}{3}-\frac{1}{q}\right)(1-a)+\left(\frac{k}{3}-\frac{1}{r}\right)a.
\end{equation}
Particularly, when $p=q=r=2$, we have
\begin{equation}\nonumber
 \norm{\nabla^if}_{L^2}\lesssim \norm{\nabla^jf}_{L^2}^\frac{k-i}{k-j}\norm{\nabla^kf}_{L^2}^\frac{i-j}{k-j}.
\end{equation}
\begin{proof}
This is a special case of  \cite[pp. 125, THEOREM]{Nirenberg}.
\end{proof}
\end{Lemma}

\begin{Lemma}[\cite{Kenig}]\label{es-product}
For any integer $k\ge1$, we have
 \begin{equation}\nonumber
  \norm{\nabla ^k(fg)}_{L^p} \lesssim \norm{f}_{L^{p_1}}\norm{\nabla ^kg}_{L^{p_2}} +\norm{\nabla ^kf}_{L^{p_3}}\norm{g}_{L^{p_4}},
 \end{equation}
and
 \begin{equation}\nonumber
  \norm{\nabla ^k(fg)-f\nabla^kg}_{L^p} \lesssim \norm{\nabla f}_{L^{p_1}}\norm{\nabla ^{k-1}g}_{L^{p_2}} +\norm{\nabla ^kf}_{L^{p_3}}\norm{g}_{L^{p_4}},
 \end{equation}
where $p, p_1, p_{2}, p_{3}, p_{4} \in[1, \infty]$ and
$$
\frac{1}{p}=\frac{1}{p_{1}}+\frac{1}{p_{2}}=\frac{1}{p_{3}}+\frac{1}{p_{4}}.
$$
\end{Lemma}

\section*{Acknowledgments}
Guochun Wu's research was partially supported by National Natural
Science Foundation of China $\#$11701193, $\#$11671086, Natural
Science Foundation of Fujian Province $\#$ 2018J05005,
$\#$2017J01562 and Program for Innovative Research Team in Science
and Technology in Fujian Province University Quanzhou High-Level
Talents Support Plan $\#$2017ZT012.   L. Yao's research  was
partially supported by Natural Science Basic Research Plan for
Distinguished Young Scholars in Shaanxi Province of China (Grant No.
2019JC-26),  National Natural Science Foundation of China $\#$
11931013.   Yinghui Zhang' research is partially supported by
Guangxi Natural Science Foundation $\#$2019JJG110003, Guangxi
Science and Technology Plan Project $\#$2019AC20214, and National
Natural Science Foundation of China $\#$11771150.


\bigskip


\begin{thebibliography}{99}

\bibitem{Bear}J. Bear, Dynamics of fluids in porous media, environmental science
series. Elsevier, New York, 1972 (reprinted with corrections, New
York, Dover, 1988).

\bibitem{Brennen1}C.E. Brennen, Fundamentals of Multiphase Flow, Cambridge University Press, New York, 2005.

\bibitem{Bresch1} D. Bresch, B. Desjardins, J.--M. Ghidaglia, E. Grenier, Global
weak solutions to a generic two--fluid model, Arch. Rational Mech.
Anal. 196(2) (2010) 599--6293.

\bibitem{Bresch2} D. Bresch, X.D. Huang, J. Li, Global weak solutions to one--dimensional non--conservative viscous
compressible two--phase system, Commun. Math. Phys. 309(3) (2012)
737--755.

\bibitem{c1}
 H.B. Cui, W.J. Wang, L. Yao, C.J. Zhu, Decay rates of a nonconservative
 compressible generic two--fluid model, SIAM J. Math.
Anal. 48(1) (2016) 470--512.

\bibitem{Dan1}
 R. Danchin, Global existence in critical spaces for compressible
Navier--Stokes equations, Invent. Math. 141 (2000) 579--614.

\bibitem{Evans}L.C. Evans, Partial Differential Equations, Amer. Math. Soc., Providence, RI, 1998.

\bibitem{Evje1}S. Evje, Weak solutions for a gas--liquid model relevant for describing gas--kick in oil wells, SIAM J. Appl. Math.
43 (2011) 1887--1922.

\bibitem{Evje2}S. Evje, Global weak solutions for a compressible gas--liquid model with well--formation interaction, J. Differential
Equations 251 (2011) 2352--2386.

\bibitem{Evje3}S. Evje, T. Fl$\mathring{a}$tten, Hybrid flux--splitting schemes for a common two--fluid model, J. Comput. Phys. 192 (2003) 175--210.

\bibitem{Evje4}S. Evje, T. Fl$\mathring{a}$tten, Weakly implicit numerical schemes for a
two--fluid model, SIAM J. Sci. Comput. 26(5) (2005) 1449--1484.

\bibitem{Evje8}S. Evje, T. Fl$\mathring{a}$tten, On the wave structure of two--phase flow models, SIAM J. Appl. Math. 67 (2006) 487--511.

\bibitem{Evje9} S. Evje, W.J. Wang, H.Y. Wen, Global
Well--Posedness and Decay Rates of Strong Solutions to a
Non--Conservative Compressible Two--Fluid Model, Arch. Rational
Mech. Anal. {221}(3) (2016) 2352--2386.

\bibitem{Friis1}H.A. Friis, S. Evje, T. Fl$\mathring{a}$tten, A numerical study of characteristic slow--transient behavior of a compressible 2D
gas--liquid two--fluid model, Adv. Appl. Math. Mech. 1 (2009)
166--200.

\bibitem{Ishii1}M. Ishii, Thermo-Fluid Dynamic Theory of Two-Phase Flow, Eyrolles, Paris, 1975.

\bibitem{Kenig} C. Kenig, G. Ponce, G., L. Vega, Well--posedness of the initial
value problem for the Korteweg--De Vries equation, J. Am. Math. Soc.
4 (1991) 323--347.

 \bibitem{Kobayashi2} T. Kobayashi, Y. Shibata, Decay estimates of solutions for the
 equations of motion of compressible viscous and heat--conductive gases in an exterior
domain in $\mathbb R^3$, Comm. Math. Phys. 200 (1999) 621--659.

\bibitem{Mat1}A. Matsumura, T. Nishida, The initial value problem for the
equations of motion of compressible viscous and heat conductive
fluids, Proc. Japan Acad. Ser. A 55 (1979) 337--342.

\bibitem{Mat2}
A. Matsumura, T. Nishida, The initial value problem for the
equations of motion of viscous and heat--conductive gases, J. Math.
Kyoto Univ. 20(1) (1980) 67--104.

\bibitem{Nirenberg} L. Nirenberg, On elliptic partial differential equations, {Ann. Scuola Norm. Sup. Pisa}, {13} (1959) 115--162.

\bibitem{Prosperetti} A. Prosperetti, G. Tryggvason, Computational methods for
multiphase flow. Cambridge University Press, 2007.

\bibitem{Raja}K. R. Rajagopal, L. Tao, Mechanics of mixtures, Series on Advances
in Mathematics for Applied Sciences, Vol. 35, World Scientific,
1995.

\bibitem{Vasseur}A. Vasseur, H.Y. Wen, C. Yu, Global weak solution to the viscous
two--fluid model with finite energy, J. Math. Pures Appl. 125(2019)
247--282.

\bibitem{WYZ}G.C. Wu, L. Yao, Y.H. Zhang, Global well--posedness and
large time behavior of classic solution to a generic two--fluid
model in a bounded domain, Preprint, 2021.

\bibitem{Yao2} L. Yao, C.J. Zhu, Existence and uniqueness of global weak solution to a two--phase flow model with vacuum, Math.
Ann. 349 (2011) 903--928.

\bibitem{Zhang4} Y.H. Zhang, C.J. Zhu, Global existence and optimal convergence rates for the strong solutions in $H^2$
to the 3D viscous liquid--gas two--phase flow model, J. Differential
Equations 258 (2015) 2315--2338.

\bibitem{Ziemer} W. Ziemer, Weakly Differentiable Functions, Springer, Berlin, 1989.

\end{thebibliography}
\end{document}